\documentclass{article}

\usepackage[sumlimits]{amsmath}
\usepackage{amssymb}
\usepackage{amsthm}

\mathcode`:="603A

\usepackage{color}

\def\twoheaddownarrow{\rlap{$\downarrow$}\raise-.5ex\hbox{$\downarrow$}}
\def\twoheaduparrow{\rlap{$\uparrow$}\raise.5ex\hbox{$\uparrow$}}

\def\threeheaddownarrow{\rlap{\rlap{$\downarrow$}\raise-.4ex\hbox{$\downarrow$}}\raise-0.8ex\hbox{$\downarrow$}}

\def\threeheaduparrow{\rlap{\rlap{$\uparrow$}\raise.4ex\hbox{$\uparrow$}}\raise0.8ex\hbox{$\uparrow$}}

\usepackage[hidelinks]{hyperref}

\usepackage[a4paper,portrait,margin=2cm]{geometry}

\usepackage{graphicx}

\usepackage{leftindex}

\usepackage[backend=bibtex,citestyle=alphabetic,bibstyle=alphabetic,maxcitenames=3,maxbibnames=5,firstinits=true]{biblatex}

\DeclareNameAlias{default}{last-first}

\usepackage{stmaryrd} 

\usepackage{tikz}
\usetikzlibrary{arrows}
\usetikzlibrary{fadings}

\usepackage{tikz-cd}
\tikzcdset{every label/.append style = {font = \small}}

\usepackage{todonotes}

\newcommand{\hideEM}[1]{}

\usepackage{xspace}
\def\ie{i.e.\xspace}
\def\eg{e.g.\xspace}

\newcommand{\id}{\mathsf{id}}

\DeclareFontFamily{OT1}{pzc}{}
\DeclareFontShape{OT1}{pzc}{m}{it}{<->s*[1.30]pzcmi7t}{}
\DeclareMathAlphabet{\mathpzc}{OT1}{pzc}{m}{it}
\def\ct#1{\mathpzc{#1}} 

\newcommand{\RQ}{\mathbb{R}}

\newcommand{\qle}{\sqsubseteq} 
\newcommand{\qJ}{\bigsqcup} 
\newcommand{\qj}{\sqcup} 
\newcommand{\qM}{\bigsqcap}
\newcommand{\qm}{\sqcap}
\newcommand{\qT}{\otimes} 
\newcommand{\qt}{\top} 
\newcommand{\qI}{\mathsf{u}} 

\newcommand{\qwb}{\twoheaddownarrow}

\newcommand{\qtb}{\threeheaddownarrow}

\newcommand{\qLowerSet}{\mathop{\downarrow}}  

\newcommand{\A}{\ct{A}}
\newcommand{\B}{\ct{B}}
\newcommand{\Set}{\ct{Set}}
\newcommand{\CL}{\ct{CL}}
\newcommand{\Po}{\ct{Po}}
\newcommand{\MS}{\ct{Met}}

\newcommand{\QUnif}{\ct{QU}} 
\newcommand{\QU}{{\cal U}}
\newcommand{\QV}{{\cal V}}

\newcommand{\Q}{\mathbb{Q}}

\newcommand{\Top}{\ct{Top}}

\newcommand{\R}{R}

\newcommand{\upperSetOf}[1]{\ensuremath{\mathop{\uparrow} {#1}}} 

\newcommand{\Idl}{\ensuremath{\mathrm{Idl}}} 

\newcommand{\cl}[1]{\overline{#1}}

\newcommand{\PS}{\mathsf{P}}
\newcommand{\setf}[1]{\ensuremath{\{{#1}\}}}
\newcommand{\setb}[2]{\ensuremath{\{{#1}\mid{#2}\}}}
\newcommand{\famb}[2]{\ensuremath{({#1}\mid{#2})}}

\newcommand{\CM}{\mathsf{C}}
\newcommand{\PM}{\mathsf{P}}

\newcommand{\defeq}{\stackrel{\vartriangle}{=}}
\newcommand{\defiff}{\stackrel{\vartriangle}{\iff}}

\newcommand{\M}{M}

\usepackage{etoolbox}

\newcommand{\Int}[2][]{\mathsf{int}\ifblank{#1}{}{_{#1}}(#2)} 

\theoremstyle{plain} 

\newtheorem{definition}{Definition}[section]

\newtheorem{corollary}[definition]{Corollary}
\newtheorem{lemma}[definition]{Lemma}

\newtheorem{proposition}[definition]{Proposition}

\newtheorem{theorem}[definition]{Theorem}

\theoremstyle{definition}

\newtheorem{example}[definition]{Example}
\newtheorem{remark}[definition]{Remark}

\newcommand{\RD}{\mathsf{RD}}

\newcommand{\rTo}{\longrightarrow}
\newcommand{\rInto}{\hookrightarrow}

\begin{document}

\title{Metrization of Quasi-Uniformities, Powerset Monads, and
  Qualitative Robustness Analysis}

\author{Francesco Dagnino\thanks{DIBRIS, University of Genoa,
    Italy, Email:\href{mailto:francesco.dagnino@unige.it}{\texttt{francesco.dagnino@unige.it}}, ORCID: 0000-0003-3599-3535}
    \and
    Amin Farjudian\thanks{School of
      Mathematics, University of Birmingham, United Kingdom, Email:
    \href{mailto:A.Farjudian@bham.ac.uk}
    {\texttt{A.Farjudian@bham.ac.uk}}, ORCID: 0000-0002-1879-0763, }
  \and
Eugenio Moggi\thanks{DIBRIS, University of Genoa,
    Italy, Email:\href{mailto:moggi@unige.it}{\texttt{moggi@unige.it}}, ORCID: 0000-0001-8018-6543}
}

\date{}

\maketitle

\begin{abstract}
  We study the relationship between quasi-uniform spaces, topological
  spaces, and quantale-valued metric spaces.  Our main result is a
  metrization theorem establishing an equivalence between the category
  of quasi-uniform spaces and a category of quantale-valued metric
  spaces. We also obtain a quantale-based metrization theorem for
  arbitrary topological spaces that refines existing
  constructions. These results identify quasi-uniformities as the
  appropriate qualitative counterpart of quantale-valued metrics.
  Building on this correspondence, we show that the Hausdorff-Smyth
  monad on quantale-valued metric spaces, which is used in
  quantitative robustness analysis, arises as a lifting of a
  corresponding monad on quasi-uniform spaces along the
  equivalence. This provides a unified categorical framework
  connecting topology, quasi-uniformity, and quantitative robustness
  analysis.
\end{abstract}


  \section{Introduction}

  Notions of metric (or distance) and continuity are fundamental in
  quantitative analysis, including robustness analysis.  The setting
  of classical metric spaces can be too restrictive, because many
  topological spaces of practical importance are not metrizable. As a
  result, various generalizations of classical metric spaces have been
  proposed to enable quantitative analysis on spaces that are not
  classically
  metrizable~\parencite[Chapter~E]{Hart_et_al:Encyclopedia_Topology:2003}. One
  such generalization is provided by \emph{quantale-valued} metric
  spaces~\parencite{Flagg:Quantales_continuity_spaces:1997,CookW21,Dagnino_Farjudian_Moggi:Robust_Topology:arXiv:2025},
  whereby some axioms of the classical metrics have been relaxed
  ({\eg}, symmetry is not required) and metrics take values in
  (continuous) quantales, which include the quantale of the extended
  non-negative real numbers
  (Definition~\ref{def:Quantale_valued_metric_spaces}).

One of the main motivations for this paper comes from \emph{robustness
analysis}, which studies how (mathematical models of) systems react to
perturbations~\parencite{Moggi_Farjudian_Duracz_Taha:Reachability_Hybrid:2018,Farjudian_Moggi:Robustness_Scott_Continuity_Computability:2023,Dagnino_Farjudian_Moggi:Robust_Topology:arXiv:2025}.
Metrics provide a natural mechanism for measuring perturbations, and
as a result, common approaches to robustness analysis rely on a metric
structure.
In particular we identify robust analyses with continuous maps between
certain \emph{hyperspaces}, namely topological spaces on subsets of a
powerset $\PS(X)$.
The hypertopologies relevant for robust analyses, called \emph{robust
  topologies}, are determined by metrics on the set $X$.  In general,
\emph{topologically equivalent} metrics on $X$ (\ie, those generating
the same open ball topology on $X$) may determine different robust
topologies on $\PS(X)$.  Only in special cases (\eg, when the metric
space $X$ is compact) the robust topology depends only on the open
ball topology generated by the metric on $X$, and is related to known
topologies, namely the Scott topology and the upper Vietories
topology~\parencite[Theorem
A.4]{Moggi_Farjudian_Duracz_Taha:Reachability_Hybrid:2018}. Therefore,
replacing quantale-valued metric spaces with topological spaces is not
a viable option for a \emph{qualitative} robustness analysis.

In~\parencite[Theorem~6.9]{Dagnino_Farjudian_Moggi:Robust_Topology:arXiv:2025},
we show that the robust topology on $\PS(X)$, determined by a
continuous quantale-valued metric on $X$, is the open ball topology
for a continuous quantale-valued metric on $\PS(X)$, called the
Hausdorff-Smyth metric. Moreover, this construction on metrics extends
to a monad, called the Hausdorff-Smyth monad, on the category of
continuous quantale-valued metric spaces and \emph{uniformly
  continuous} maps.

In this paper we show that the category of quasi-uniform spaces and
quasi-uniformly continuous
maps~\parencite{Fletcher_Lindgren:Quasi_Uniform:Book:1982} provides the
correct qualitative framework for robustness analysis, whereby
continuous quantale-valued metrics are replaced with structures
(called \emph{quasi-uniformities}) of binary relations (called
\emph{entourages}), see Definition~\ref{def:quasi_unif}.  These
structures retain enough information to reason about uniform
continuity (and continuity).

More generally, we investigate the relationship between the
qualitative and quantitative perspectives.  In order to make these
concepts precise, we work within the framework of \emph{topological
functors} (see Definition~\ref{def:TF}), identifying as qualitative
those having certain strictness properties (see
Remark~\ref{rem:ms-equiv}).
Our starting point is the observation that continuous quantale-valued
metric spaces induce both a topology and a quasi-uniformity
(Definition~\ref{def:Phi}), yielding a network of forgetful functors
between several categories, as depicted in
Diagram~\eqref{diag:Met_QU_Top}.
While the passage from metrics to quasi-uniformities and topologies is
straightforward, the converse direction---the reconstruction of a
metric-like structure from qualitative data---is more subtle, and
amounts to proving that certain forgetful functors are equivalences.
Finally, Theorem~\ref{thm:HS_Lifting_PS_monad}, which states that the
Hausdorff-Smyth monad on metric spaces is a \emph{lifting} of one of
three Hausdorff-like monads on quasi-uniform spaces, shows that
quasi-uniformities are suitable for robustness analysis.

\paragraph{Contributions.} The main contribution of this paper is to
establish (at a category-theoretic level) connections between the
qualitative frameworks of topological spaces and quasi-uniform spaces
on the one hand, and the quantitative framework of continuous
quantale-valued metric spaces on the other hand. This shows that, in
the context of robustness analysis, quasi-uniform spaces provide the
right qualitative abstraction over metric spaces, and continuous
quantale-valued metric spaces provide a quantitative realization. In
more details, our contributions are as follows:

\begin{description}
\item[Relationship between topological and quasi-uniform spaces:] We
  show that the category of topological spaces is a coreflective
  subcategory of the category of quasi-uniform spaces
  (Theorem~\ref{thm:las:T}). In particular, each quasi-uniformity on
  $X$ induces a topology on $X$ and, for every topology $\tau$ on $X$,
  there exists the largest quasi-uniformity whose induced topology
  coincides with $\tau$.
  
\item[Metrization of quasi-uniform spaces:] Our central result
  (Theorem~\ref{thm:met:qu}) shows that the category of quasi-uniform
  spaces is equivalent to a category of quantale-valued metric
  spaces and uniformly continuous maps. This is achieved by
  constructing, from a base for a quasi-uniformity
  (Definition~\ref{def:QUnif_Base_Subbase}), a prime-continuous
  quantale (Definition~\ref{def:continuous_prime_cont_latt}) via a
  rounded-lowerset construction (Section~\ref{subsec:Rounded_Lowersets}).

\item[Metrization of topologies via quantales:] We give a construction
  showing that the category of topological spaces is equivalent to a
  category of quantale-valued metric spaces and continuous maps
  (Theorem~\ref{thm:met:top}). This refines known results
  ({\eg}, \parencite{Flagg:Quantales_continuity_spaces:1997}) by
  producing smaller and more manageable quantales.

\item[Powerset-like monads:] Building on the above equivalences, in
  Section~\ref{sec:Powerset-like-Monads}, we give three liftings of
  the powerset monad on the category of sets along the topological
  functor from the category of quasi-uniform spaces (and, thus, from
  the equivalent category of metric spaces) into the category of
  sets. This provides a unified categorical account of Hausdorff-type
  constructions in both quasi-uniform and metric settings.
\end{description}

\subsection{Structure of the Paper}
\label{subsec:structure}

After reviewing the necessary categorical, order-theoretic, and metric
preliminaries in Section~\ref{sec:preliminaries}, we introduce
quasi-uniform spaces and their relation to topological spaces in
Section~\ref{sec:q_unif_spaces}.  Section~\ref{sec:metrization}
includes the metrization results, showing that quasi-uniform
spaces and topological spaces can be metrized by continuous quantale-valued
metrics. Section~\ref{sec:Powerset-like-Monads} applies these results
to powerset constructions and Hausdorff-type monads, culminating in
the connection with the Hausdorff-Smyth monad used in robustness
analysis. Finally, Section~\ref{sec:concluding_remarks} contains some
concluding remarks.

\section{Preliminaries}
\label{sec:preliminaries}

We write $\PS(X)$ for the set of subsets of $X$, $\PS_f(X)$ for the
set of finite subsets of $X$, and $A \subseteq_f A'$ when $A$ is a
\emph{finite} subset of $A'$.
We denote by $\omega$ the set of natural numbers, and identify a
natural number $n\in\omega$ with the set of its predecessors.

We assume basic familiarity with Category
Theory~\parencite{borceux1994handbook1}.  We write $\Set$ for the
category of sets and (total) maps, $\Po$ for the category of preorders
and monotone maps, and $\Top$ for the category of topological spaces
and continuous maps. We use the symbol `$\in$' for set membership
({\eg}, $x \in X$), but we use `$:$' for membership of function types
(\eg, $f: X \to Y$) and to denote objects and arrows in categories
(\eg, $X:\Top$ and $f:\Top(X,Y)$). We write $(a_i\in A_i\mid i\in I)$
for an $I$-indexed family of elements, \ie, an element in the product
$\prod_{i\in I} A_i$, or simply $(a_i\mid i\in I)$ when the
$I$-indexed family of sets $(A_i\mid i\in I)$ is clear from the
context.

\subsection{Topological Functors}

Most of the categories we consider in this paper have a
\emph{topological functor} into $\Set$.  We recall the definition of
topological functor and some of their properties (\eg, see~\parencite[Sec
  7.3]{borceux1994handbook2}).

  \begin{definition}[Topological Functor]
    \label{def:TF}
    A functor $U:\A\to\B$ is called \emph{topological} $\defiff$ for
    every $B : \B$ and class $I$, if $\famb{A_i}{i\in I}$ is a family
    of objects in $\A$ and $\famb{f_i:B\to U(A_i)}{i\in I}$ is a
    family of arrows in $\B$ (called $I$-cone from $B$ in $\B$), then
    there exist $A : \A$ and a family $\famb{g_i:A\to A_i}{i\in I}$ of
    arrows in $\A$ (\ie, an $I$-cone from $A$ in $\A$) such that
  \begin{itemize}
  \item $U(A)=B$ and $\forall i\in I.U(g_i)=f_i$
  \item if $\famb{g'_i:A'\to A_i}{i\in I}$ is an $I$-cone from $A'$ in
    $\A$ such that $\forall i\in I.U(g'_i)=f_i\circ f$ for some
    $f:U(A')\to B$ in $\B$, then there exists a unique $g:A'\to A$ in
    $\A$ such that $U(g)=f$ and $\forall i\in I.g'_i=g_i\circ g$.
  \end{itemize}
  The pair $(A,\famb{g_i:A\to A_i}{i\in I})$ is called an
  \emph{initial structure} for
  $(B,\famb{A_i}{i\in I},\famb{f_i:B\to U(A_i)}{i\in I})$.
  The topological functor $U$ is called
  \begin{itemize}
  \item \emph{strict} when every $(B,\famb{A_i}{i\in I},\famb{f_i:B\to
    U(A_i)}{i\in I})$ has a unique initial structure,
  \item \emph{small} when for every $B\in\B$ the class
    $\famb{A\in\A}{U(A)=B}$ is a set.
  \end{itemize}
\end{definition}

The most relevant properties of topological functors are summarized in
the following proposition.
\begin{proposition}\label{prop:TF}
  If $U:\A\to\B$ is a topological functor, then:
  
  \begin{enumerate}
  \item $U^{op}:\A^{op}\to\B^{op}$ is a topological functor too.
  \item $U$ is faithful, surjective on objects, and creates limits and
    colimits.
  \item $U$ has full\&faithful left- and right-adjoints, denoted $L$
    and $R$, respectively. Moreover, they can be chosen to be
    \emph{sections}, \ie, $U\circ L=U\circ R=\id_\B$.
  \item For every $B\in\B$ the fiber $\A_B$, \ie, the sub-category of
    $\A$ whose arrows are the $g$ such that $U(g)=\id_B$, is a
    \emph{complete preorder}, \ie, it has all joins (including large ones).
  \item If $U$ is also strict, then the fibers are posets.
  \item If $U$ is also small, then the fibers are small.
  \end{enumerate}
\end{proposition}

  \begin{example}
    \label{example:TF}
    The forgetful functors $U:\Top\to\Set$ and $U:\Po\to\Set$, namely
    $U(X,\tau)=U(X,\qle)\defeq X$ and $U(f)=f$ ({\ie}, on arrows they
    are hom-sets inclusion), are examples of strict and small
    topological functors.
The restriction of these functors to the full sub-categories
$\Top_0$ of $T_0$-topological spaces and $\Po_0$ of posets are not
topological, because they do not have right-adjoints.
\end{example}

\subsection{Preorders, Lattices, and Quantales}

  Given a preorder $(X,\qle)$ its \emph{dual} is $(X,\qle^o)$, where
  $x\qle^o y\defiff y\qle x$. The definition of dual applies to any
  binary relation on $X$ (see Section~\ref{sec:BR}). The dual induces
  an \emph{involution} on $\Po$, \ie, an endofunctor $-^o:\Po\rTo \Po$
  which is its own inverse, whose action on arrows is $f^o=f$.

\begin{definition}
Given a preorder $(X,\qle)$ and a subset $A\subseteq X$, we say that:
\begin{itemize}
\item $A$ is an \emph{upperset} $\defiff\forall x \in A. \forall y \in
  X. \
  x\qle y\implies y\in A$.
\item $A$ is \emph{filtered} $\defiff$ $A$ is non-empty and
  $\forall x,y\in A.\exists z\in A.z\qle x,y$.
\end{itemize}
A \emph{lowerset} and a \emph{directed} subset in $(X,\qle)$ are
defined as an upperset and a filtered subset in the dual $(X,\qle^o)$.
\end{definition}

\begin{definition}
  A \emph{complete lattice} $(Q,\qle)$ is a poset in which every
  subset $A\subseteq Q$ has a \emph{join} $\qJ A$ (and as a
  consequence, also a \emph{meet} $\qM A$).  A \emph{quantale}
  $(Q,\qle,\qT,\qI)$ is a monoid $(Q,\qT,\qI)$ on a complete lattice
  $(Q,\qle)$ satisfying the following distributivity law:
  \begin{equation*}
     \forall A,B\subseteq Q.\quad
    \left(\qJ A\right) \qT \left( \qJ B \right) = \qJ\setb{x\qT y}{x\in A\land y\in B}.    
  \end{equation*}
  Given a quantale $(Q,\qle,\qT,\qI)$, its \emph{dual} is the quantale
  $(Q,\qle,\qT^o,\qI)$, where $x\qT^oy\defeq y\qT x$.
\end{definition}
If $(Q,\qle)$ is a complete lattice, then so is $(Q,\qle^o)$. This is
false for quantales, \ie, if $(Q,\qle,\qT,\qI)$ is a quantale, then
$(Q,\qle^o,\qT,\qI)$ fails to satisfy distributivity (unless $Q$ is
trivial).
However, on the wider class of preordered monoids (\ie, monoids in
$\Po$), both $(Q,\qle^o,\qT,\qI)$ and $(Q,\qle,\qT^o,\qI)$ are
preordered monoids.

  \begin{definition}
    \label{def:monoidal:maps}
    Given two quantales (more generally, two preordered monoids)
    $(Q,\qle,\qT,\qI)$ and $(Q',\qle',\qT',\qI')$, a monotone map
    $f:(Q,\qle)\to (Q',\qle')$ is
  \begin{itemize}
  \item \emph{lax-monoidal} $\defiff$ $\qI'\qle' f(\qI)$ and $\forall
    q_0,q_1\in Q.f(q_0)\qT' f(q_1)\qle' f(q_0\qT q_1)$
  \item \emph{strict-monoidal} $\defiff$ $\qI'= f(\qI)$ and $\forall
    q_0,q_1\in Q.f(q_0)\qT' f(q_1)=f(q_0\qT q_1)$.
  \end{itemize}
\end{definition}
The map $(Q,\qle,\qT,\qI)\mapsto(Q,\qle,\qT^o,\qI)$ on preordered
monoids extends to an involution on the category of preordered monoids
and lax-monoidal maps, and restricts to an involution on the
sub-category of preordered monoids and strict monoidal maps.
The map $(Q,\qle,\qT,\qI)\mapsto(Q,\qle^o,\qT,\qI)$ on preordered
monoids extends to an involution only on the sub-category of
preordered monoids and strict monoidal maps.

The following result identifies inequalities preserved by lax-monoidal
maps (between preordered monoids).
\begin{proposition}
  \label{prop:lax}
  If a map $f: (Q,\qle,\qT,\qI) \to (Q',\qle',\qT',\qI')$ is
  lax-monoidal, then
  \begin{equation*}
    (\qT_{i\in n} q_i)\qle q \implies (\qT'_{i\in n} f(q_i))\qle' f(q),   
  \end{equation*}
  where $(\qT_{i\in 0} q_i)\defeq\qI$ and $(\qT_{i\in (n+1)} q_i)\defeq
  (\qT_{i\in n} q_i)\qT q_n$.  Moreover, we write $q^n$ for
  $(\qT_{i\in n} q)$.
\end{proposition}
\begin{proof}
  We prove  by induction on $n$ that
  $(\qT'_{i\in n} f(q_i))\qle' f(\qT_{i\in n} q_i)$.
  For the base case $0$ we have $\qI'\qle'f(\qI)$, because $f$ is lax-monoidal.
  For the inductive step $n+1$ we have:
  \begin{itemize}
   \item $(\qT'_{i\in (n+1)} f(q_i))=$ by definition
   \item $(\qT'_{i\in n} f(q_i))\qT' f(q_n)\qle'$ by the induction
     hypothesis and monotonicity of $\qT'$
   \item $f(\qT_{i\in n} q_i)\qT' f(q_n)\qle'$ by $f$ lax-monoidal
   \item $f(\qT_{i\in (n+1)} q_i)$.
  \end{itemize}
  Therefore, by $(\qT_{i\in n} q_i)\qle q$ and monotonicity of $f$, we
  get $(\qT'_{i\in n} f(q_i))\qle' f(\qT_{i\in n} q_i)\qle' f(q)$.
\end{proof}

Strict-monoidal maps preserve a wider class of inequalities, namely
$(\qT_{i\in m} q_i)\qle (\qT_{j\in n} q'_j)$.  However, most of the
maps we consider are only lax-monoidal.

\begin{definition}
Given two complete lattices $(Q,\qle)$ and $(Q',\qle')$ and a monotone
map $f:Q\to Q'$, we say that:
\begin{itemize}
\item $f$ is \emph{Scott continuous} $\defiff$ $\qJ f(D)=f(\qJ D)$,
  for every directed $D\subseteq Q$.
\item $f$ is \emph{join-preserving} $\defiff$ $\qJ f(D)=f(\qJ D)$,
  for every $D\subseteq Q$.
\end{itemize}
\end{definition}

The category $\Po$ is a cartesian closed category (CCC), and
restricting to complete lattices (and monotonic maps) yields a full
sub-CCC.  However, a more interesting CCC is the category $\CL$ of
complete lattices and Scott continuous maps. We focus on the full
sub-CCC of $\CL$ consisting of continuous lattices
(see~\parencite{AbramskyJung94-DT,Gierz-ContinuousLattices-2003,Goubault-Larrecq:Non_Hausdorff_topology:2013})
and on \emph{continuous quantales}
(see~\parencite[Sec~3.1]{Dagnino_Farjudian_Moggi:Robust_Topology:arXiv:2025}).
\begin{definition}
Given a complete lattice $(Q,\qle)$ and $x,y\in Q$, we say that:
\begin{enumerate}
\item $x$ is \textbf{way-below} $y$ (notation $x\ll y$) $\defiff$
  $y\qle\qJ D\implies\exists d \in D.x\qle d$, for every directed
  $D\subseteq Q$.
\item $x$ is \textbf{totally-below} $y$ (notation $x\lll y$) $\defiff$
  $y\qle\qJ D\implies\exists d \in D.x\qle d$, for every
  $D\subseteq Q$.
\end{enumerate}
We write $\qwb y$ and $\qtb y$ to denote the sets
$\setb{x \in Q}{x\ll y}$ and $\setb{x \in Q}{x\lll y}$, respectively.
\end{definition}

The relations $\ll$ and $\lll$ have the following properties:
\begin{proposition}
  \label{prop:orders_and_approx}
  Given a complete lattice $(Q,\qle)$ and $x, x_0, x_1 \in Q$, one has:
  \begin{enumerate}
  \item $x_0\lll x_1\implies x_0\ll x_1\implies x_0\qle x_1$.
  \item $x_0'\qle x_0\ll x_1\qle x_1'\implies x_0'\ll x_1'$.
  \item \label{item:order_totally_below}
    $x_0'\qle x_0\lll x_1\qle x_1'\implies x_0'\lll x_1'$.
  \item $\bot\ll x$ and $\qwb x$ is directed. More precisely,
    $x_0,x_1\ll x\implies x_0\qj x_1\ll x$.
  \end{enumerate}
\end{proposition}

\begin{definition}
  \label{def:continuous_prime_cont_latt}
  Given a complete lattice $(Q,\qle)$ or a quantale
  $(Q,\qle,\qT,\qI)$, we say that:
  \begin{enumerate}
  \item $Q$ is \textbf{continuous} $\defiff\forall y \in Q.y=\qJ\qwb
    y$.
  \item $B\subseteq Q$ is a \textbf{base} for $Q$ $\defiff\forall y \in
    Q.\qwb_B y$ is directed and $y=\qJ\qwb_B y$, where $\qwb_B
    y\defeq\setb{x \in B}{x\ll y}$.

  \item $Q$ is \textbf{prime-continuous} $\defiff\forall
    y \in Q.y=\qJ \qtb y$.
  \item $B\subseteq Q$ is a \textbf{prime-base} for $Q$
    $\defiff\forall y \in Q.  y=\qJ \qtb_B y$, where
    $\qtb_B y \defeq \setb{x \in B}{x\lll y}$.
  \end{enumerate}
\end{definition}

A fundamental feature of the way-below and totally-below relations is
the \emph{interpolation} property:
    
\begin{proposition}[Interpolation
    {\parencite[Lemma~2.2.15]{AbramskyJung94-DT}}]\label{prop:int:1}\
  \begin{itemize}
  \item If $B$ is a base for the continuous lattice $Q$, then $q_1 \ll
    q_2\implies\exists q \in B.q_1\ll q\ll q_2$.
  \item If $B$ is a prime-base for the prime-continuous lattice $Q$,
    then $q_1 \lll q_2\implies\exists q \in B.q_1\lll q\lll q_2$.
  \end{itemize}
\end{proposition}
The interpolation property for continuous lattices given above is
equivalent to that for continuous domains given
in~\parencite[Lemma~2.2.15]{AbramskyJung94-DT}, because in a continuous
lattice $(\forall q'\in A.q'\ll q)\iff(\qJ A)\ll q$ for every
$A\subseteq_f Q$, while in a domain the finite join $\qJ A$ may not
exists.

\begin{proposition}[{\parencite[Lemma~3.19]{Dagnino_Farjudian_Moggi:Robust_Topology:arXiv:2025}}]\label{prop:int:2}\
  \begin{itemize}
  \item If $B$ is a base for the quantale $Q$, then
    $q \ll q_1\qT q_2\implies \exists q'\in B.q'\ll q_2\land q \ll q_1 \qT q'$.
  \item If $B$ is a prime-base for the quantale $Q$,
    then $q \lll q_1\qT q_2\implies \exists q'\in B.q'\lll q_2\land q
    \lll q_1 \qT q'$.
  \end{itemize}
  Both claims hold also when $\qT$ is replaced by its dual $\qT^o$.
\end{proposition}

The following result relates bases and prime-bases.

\begin{proposition}\label{prop:base:pbase}
  If $P$ is a prime-base and $B$ is a base for a prime-continuous
  lattice $(Q,\qle)$, then
  \begin{enumerate}
  \item $B$ is a prime-base for $Q$
  \item $P_\qj$ is a base for $Q$, where $P_\qj$ is the closure of $P$
    under finite joins computed in $Q$.
  \end{enumerate}
\end{proposition}

\begin{proof}
  For the first claim, we prove $\forall y\in Q.y=\qJ\setb{b\in
    B}{b\lll y}$. Let $S\defeq\setb{b\in B}{\exists x\in Q.b\ll x\lll
    y}$, then $\qJ S=\qJ_{x\lll y}(\qJ\qwb_B x)=\qJ_{x\lll
    y}x=y=\qJ\qwb_By$. But $S\subseteq\setb{b\in B}{b\lll
    y}\subseteq\qwb_By$, thus $y=\qJ\setb{b\in B}{b\lll y}$.

  For the second claim, we prove $\forall y\in Q.\qwb_{P_\qj}y$ is
  directed and $y=\qJ\qwb_{P_\qj}y$. The set $\qwb_{P_\qj}y$ is
  directed because $P_\qj$ is closed under finite joins.  Let
  $S\defeq\setb{p\in P}{p\lll y}$. Then, $\qJ S=y=\qJ\qwb y$. But
  $S\subseteq\qwb_{P_\qj}y\subseteq\qwb y$, thus $y=\qwb_{P_\qj}y$.
\end{proof}

\begin{example}
  The following examples demonstrate proper inclusion between some
  lattice-related concepts:
  \begin{enumerate}
  \item The diamond lattice $M_3 \defeq \setf{\bot, x, y, z, \top}$
    with  $\bot<x,y,z<\top$ is a continuous lattice, but not
    prime-continuous, because it is not distributive.

  \item The poset $O(\Q)$ of open subsets of the rational numbers
    (with the Euclidean topology) ordered by inclusion is a locale,
    but not a continuous lattice, because $\Q$ is not locally compact.
    
  \item Recall that an open set $W$ is said to be regular if it is the
    interior of its closure, {\ie}, $W = \Int{\cl{W}}$. The
    poset $RO([0,1])$ of the regular open subsets of the unit interval
    ordered by inclusion is a complete boolean algebra, thus a
    quantale, which is continuous, but not prime-continuous.
  \end{enumerate}
\end{example}

Products of (prime-)continuous lattices are (prime-)continuous; the
complete lattice of Scott continuous maps between (prime-)continuous
lattices is (prime-)continuous.  Thus (prime-)continuous lattices form
full sub-CCCs of the category $\CL$ of complete lattices and
Scott continuous maps.
It is well-known that continuous domains coincide with the Scott
continuous retracts of algebraic domain~\parencite[Ch VII, Sec
  2.3]{Johnstone86}.  When domains are replaced by complete lattices,
a stronger result holds, namely continuous lattices coincide with the
Scott continuous retracts of prime-algebraic lattices.

\begin{proposition}
  In the category $\CL$ of complete lattices and Scott continuous
  maps, every continuous lattice is a retract of a prime-algebraic
  lattice.
\end{proposition}
\begin{proof}
  If $B$ is a base for the continuous lattice $(X,\qle)$ and
  $\Idl(B,\qle)$ denotes the ideal completion of the join-semilattice
  $(B,\qle)$, then $\qwb_B:(X,\qle)\to \Idl(B,\qle)$ has a
  right-adjoint $R(I) \defeq \qJ I$ in $\CL$ (see
  also~\parencite[Proposition~3.1.6]{AbramskyJung94-DT}).
  The inclusion $\Idl(B,\qle)\to \PS(|B|)$ has a left-adjoint in $\CL$,
  namely, $L(A) \defeq$ the ideal in $(B,\qle)$ generated by $A$.
  Hence, $(X,\qle)$ is a retract of the prime-algebraic lattice
  $\PS(|B|)$.
\end{proof}

\subsection{Topologies and Metric Spaces}

\begin{definition}
  \label{def:top_base_subbase}
  Given a topological space $(X,\tau)$ and a set $B\subseteq\tau$ of
  its open subsets, we say that:
\begin{itemize}
\item $B$ is a \emph{base} for $\tau$ $\defiff$ every open set in
  $\tau$ is the union of a family of open sets in $B$.
\item $B$ is a \emph{sub-base} for $\tau$ $\defiff$ $\tau$ is the
  smallest topology generated by (\ie, containing) $B$, or
  equivalently, the closure of $B$ under finite intersections is a
  base for $\tau$.
\end{itemize}
Given $A\subseteq X$, we write $\Int[\tau]{A}$ for the interior of
$A$, \ie, the biggest $O\in\tau$ contained in $A$, and omit the
subscript $\tau$ when the topology is clear from the context.

The \emph{specialization preorder} $\leq_\tau$ on $X$ induced by
$\tau$ is $x\leq_\tau y\defiff\forall O\in B.x\in O\implies y\in O$,
where $B$ is any sub-base for $\tau$.
\end{definition}

\begin{remark}
In Definition~\ref{def:top_base_subbase},  the preorder $\leq_\tau$
does not depend on the choice of sub-base $B$. In fact, $x \leq_\tau y
\iff x \in \cl{\setf{y}}$.
\end{remark}

Diagram~\eqref{diag:forget1} below describes the relations between
$\Top$, $\Po$ and $\Set$, where:
\begin{itemize}
\item $U:\Top\to\Set$ and $U:\Po\to\Set$ are the topological functors
  defined in Example~\ref{example:TF}.
\item $P:\Top\to\Po$ is the functor defined by
  $P(X,\tau)\defeq(X,\leq_\tau)$ and on arrows is hom-sets inclusion.
\item $L_P:\Po\to\Top$ is the functor defined by
  $L_P(X,\leq)\defeq (X,\tau_\leq)$, where $\tau_\leq$ is the set of
  uppersets in $(X,\leq)$ (also known as the Alexandroff topology),
  and on arrows is hom-sets inclusion.
\end{itemize}

\begin{equation}
\begin{tikzcd}[row sep = large, column sep = huge]
  \Top \arrow[r, "P"'] \arrow[rd, "U"'] & \Po
  \arrow[l,bend right,dashed,"\bot","L_P"'] \arrow[d, "U"]\\
  &\Set
\end{tikzcd}
\label{diag:forget1}
\end{equation}
The functor $P$ is faithful and $L_P$ is a left-adjoint to $P$ and
also a $P$-section (\ie, $P\circ L_P=\id_\Po$). Since $L_P$ is
full\&faithful, it follows that $\Po$ is (isomorphic to) a
coreflective sub-category of $\Top$.  The functor $P$ does not have a
right-adjoint. Therefore, it is not topological.

We recall the definition of a quantale-valued metric space (see
\parencite{Dagnino_Farjudian_Moggi:Robust_Topology:arXiv:2025}), and define
the categories $\MS_c$ and $\MS_u$ with continuous and uniformly
continuous maps, respectively. We do this by mimicking the standard
definition of (uniformly) continuous map using epsilon-delta in the
context of continuous quantales.

\begin{definition}[Quantale-valued Metric Space]
  \label{def:Quantale_valued_metric_spaces}
  A \emph{quantale-valued metric space} is a triple $(X,d,Q)$
  consisting of a set $X$, a continuous quantale $(Q,\qle,\qT,\qI)$
  and a map $d:X^2\to Q$, called a \emph{metric}, such that
  \begin{equation*}
    \forall x\in X.\ \qI\qle d(x,x)\quad\mbox{and}\quad
  \forall x,y,z\in Q.\ d(x,y)\qT d(y,z)\qle d(x,z).
  \end{equation*}
  Its \emph{dual} is the metric space $(X,d^o,Q^o)$, where $Q^o$ is
  the dual of the continuous quantale $Q$ and $d^o(x,y)\defeq d(y,x)$.

  Given two quantale-valued metric spaces $(X,d,Q)$ and $(X',d',Q')$
  and a map $f:X\to X'$, we say that:
  \begin{itemize}
  \item $f$ is \emph{continuous} $\defiff\forall x\in
    X.\forall\epsilon\ll'\qI'.\exists\delta\ll\qI.  \forall y\in
    X.\delta\ll d(x,y)\implies\epsilon\ll' d'(f(x),f(y))$.
  \item $f$ is \emph{uniformly continuous}
    $\defiff\forall\epsilon\ll'\qI'.\exists\delta\ll\qI.  \forall
    x,y\in X.\delta\ll d(x,y)\implies\epsilon\ll' d'(f(x),f(y))$.
  \end{itemize}
  $\MS_c$ denotes the category of quantale-valued metric spaces and
  continuous maps, and $\MS_u$ denotes the sub-category of
  quantale-valued metric spaces and uniformly continuous maps.
\end{definition}

For every metric space $(X, d, Q)$, it is straightforward to verify
that its dual $(X,d^o,Q^o)$ is also a metric space.  This
dual induces an involution on $\MS_u$, but not on $\MS_c$.
\begin{proposition}
  \label{prop:Met_u_dual_involution}
  The dual on metric spaces extends to an involution on $\MS_u$ such that
  $f^o=f$.
\end{proposition}
\begin{proof}
  Assume that $f: \MS_u( (X_1, d_1, Q_1), ( X_2, d_2, Q_2))$. By
  definition of uniform continuity, we have:
  \begin{equation*}
    \forall\epsilon\ll_2 \qI_2.\exists\delta\ll_1\qI_1.\forall
    x,y\in X_1.\ \delta\ll_1 d_1(x,y)\implies\epsilon\ll_2 d_2(f(x),f(y)).
  \end{equation*}
  By the equivalences $\delta \ll_1 d_1^o( x, y) \iff \delta \ll_1
  d_1( y, x)$ and $\epsilon\ll_2 d_2^o(f(x),f(y)) \iff \epsilon\ll_2
  d_2(f(y),f(x))$, we conclude $f: \MS_u( (X_1, d_1^o, Q_1^o), ( X_2,
  d_2^o, Q_2^o))$.
\end{proof}

\begin{example}
  \label{example:Met_c_dual_no_involution}
  Consider the metric space $M \defeq ( [0,1], d, \RQ_+)$, in which:
  \begin{itemize}
  \item $\RQ_+$ is the quantale $([0,+\infty],\geq,+,0)$ on the
    extended non-negative reals~\parencite{lawvere1973metric};
  \item $\forall x,y \in [0,1]: \ d( x, y) \defeq \max( x-y, 0)$.
  \end{itemize}
  The function $f: M \to M$ such that $f(x)=\mbox{$0$ if $x=0$ else
    $1$}$ is continuous, \ie, $f : \MS_c( M, M)$, but $f=f^o$ is not
  continuous as a map on $M^o = ( [0,1], d^o, \RQ_+)$, where $d^o( x,
  y) = d( y, x)$.
  More precisely, $f: M^o \to M^o$ is not continuous at $x = 0$. To that
  end, take $\epsilon\in (0,1)$, then for any $\delta \in ( 0, 1]$ and
  $y \in (0, \delta)$, we have $d^o( 0, y) = d( y, 0) = y < \delta$,
  but $d^o( f( 0), f( y)) = d( f( y) , f( 0)) = d( 1, 0) = 1 > \epsilon$.
\end{example}

\subsection{Quantales of Binary Relations}
\label{sec:BR}

In this section, we recall some general facts about binary relations
which will be used in subsequent sections.
\begin{proposition}\label{prop:BR}
  Given a set $X$:
  \begin{enumerate}
  \item The quadruple $(\PS(X^2),\subseteq,\qT,\Delta)$ is a quantale,
    where
    $R_0\qT R_1\defeq\setb{(x,z)}{\exists y\in X.(x,y)\in R_0\land
      (y,z)\in R_1}$ is composition of binary relations and
    $\Delta\defeq \setb{(x,x)}{x\in X}$ is the diagonal relation.

  \item The map $-^o:\PS(X^2)\to\PS(X^2)$ such that
    $R^o\defeq\setb{(y,x)}{(x,y)\in R}$ is the dual of $R$, is a strict-monoidal
    isomorphism from $(\PS(X^2),\subseteq,\qT,\Delta)$ to its dual,
    whose tensor is $R_0\qT^o R_1\defeq R_1\qT R_0=(R_0^o\qT
    R_1^o)^o$, moreover $-^o$ preserves joins and meets, \ie, $(\cup_i
    R_i)^o=\cup_i R_i^o$ and $(\cap_i R_i)^o=\cap_i R_i^o$.

\item If all $R_i\in\PS(X^2)$ (for $i\in I$) have any of the
  following properties, then so does $\cap_i R_i$: reflexivity
  ($\Delta\subseteq R$), symmetry ($R^o\subseteq R$), transitivity
  ($R^2\subseteq R$).
  \end{enumerate}
\end{proposition}

We define several lax-monoidal maps between quantales of binary
relations.  Therefore, thanks to Proposition~\ref{prop:lax}, we get
that such maps preserve certain properties of binary relations, such
as reflexivity and transitivity.

\begin{proposition}
  \label{prop:BR:dagger}
  Given a map $f:X\to Y$, the map
  $f^\dagger\defeq(f\times
  f)^{-1}:\PS(Y^2)\to\PS(X^2)$
  is lax-monoidal.
\end{proposition}
\begin{proof}
  Since $f^\dagger$ is an inverse image map, it is monotone (and
  preserves unions and intersections).  Moreover, it is lax-monoidal
  because: $\forall x\in X.(x,x)\in
  f^\dagger(\Delta_Y)=\setb{(x,x')}{f(x)=f(x')}$ and $(x_0,x_1)\in
  f^\dagger(R_0)\land (x_1,x_2)\in f^\dagger(R_1)\iff
  (f(x_0),f(x_1))\in R_0\land (f(x_1),f(x_2))\in R_1\implies
  (f(x_0),f(x_2))\in (R_0\qT_Y R_1) \implies ( x_0, x_2 ) \in
  f^\dagger(R_0\qT_Y R_1)$.
\end{proof}

\begin{proposition}\label{prop:BR:liftS}
  The map $-_S:\PS(X^2)\to\PS(\PS(X)^2)$ such that $(A,A')\in
  R_S\defiff A'\subseteq R(A)$, where $R(A)\defeq\setb{x'\in
    X}{\exists x\in A.(x,x')\in R}$, is lax-monoidal.
\end{proposition}
\begin{proof}
  Monotonicity follows from $R\subseteq R'\implies R(A)\subseteq
  R'(A)$.
  Moreover, $-_S$ is lax-monoidal because:
  $\forall
  A,A'\in\PS(X).(A,A')\in(\Delta_X)_S\iff A'\subseteq A$ and
  $(A,A'')\in R_S\qT R'_S\iff \exists A'.A'\subseteq R(A)\land
  A''\subseteq R'(A')\implies A''\subseteq R'(R(A))=(R\qT R')(A)$.
\end{proof}

The lax-homomorphism $-_S$ is an example of a lax-homomorphism
$\PS(X^2)\to\PS(\M(X)^2)$ induced by a \emph{relator} for an
endofunctor $\M$ on $\Set$ (see~\parencite[Definition 19]{gavazzo2019phd}), in
this case $\M$ is the covariant powerset functor.  Moreover, the
following closure properties of lax-homomorphisms between quantales of
binary relations should be compared with similar closure properties of
relators for an endofunctor in~\parencite[Prop 8]{gavazzo2019phd}.

\begin{proposition}\label{prop:BR:main}\
  \begin{enumerate}
  \item If $F_i:\PS(X^2)\to\PS(Y^2)$ is lax-monoidal for every $i\in
    I$, then so is $F(R)\defeq\cap_i F_i(R)$.
  \item If $F:\PS(X^2)\to\PS(Y^2)$ is lax-monoidal (strict-monoidal),
    then so is $F^o(R)\defeq F(R^o)^o$.
  \end{enumerate}
\end{proposition}

\begin{proof}
  For the first implication, each property of $F$ follows from
  the corresponding property of $F_i$ for all $i\in I$:
  \begin{itemize}
  \item $R_0\subseteq R_1\implies F(R_0)\subseteq F(R_1)$, because
    $R_0\subseteq R_1\implies$ (by definition of $F(R_0)$ and $F_i$
    monotone)\\ $\forall i.F(R_0)\subseteq F_i(R_0)\subseteq
    F_i(R_1)\implies$ (by definition of $F(R_1)$)\\ $F(R_0)\subseteq
    F(R_1)$.

  \item $\Delta_Y\subseteq F(R)$, because $\Delta_Y\subseteq
    F(R)\iff\forall i \in I.\ \Delta_Y\subseteq F_i(R)$ by definition
    of $F$.

\item $F(R_0)\qT F(R_1)\subseteq F(R_0\qT R_1)$, because
  $\forall i \in I.\ F(R_0)\qT F(R_1)\subseteq F_i(R_0)\qT
  F_i(R_1)\subseteq F_i(R_0\qT R_1)$ by definition of $F$,
  monotonicity of $\qT$, and $F_i$ being lax-monoidal. This is
  equivalent to $F(R_0)\qT F(R_1)\subseteq F(R_0\qT R_1)$, by
  definition of $F$.
\end{itemize}

For the second implication as well, each property of $F^o$ follows
from the corresponding property of $F$:
  \begin{itemize}
  \item $R_0\subseteq R_1\implies F^o(R_0)\subseteq F^o(R_1)$, because
    $R_0\subseteq R_1\iff R_0^o\subseteq R_1^o\implies$ (by
    monotonicity of $F$)\\ $F(R_0^o)\subseteq F(R_1^o)\iff F(R_0^o)^o\subseteq
    F(R_1^o)^o$.
  \item $\Delta_Y\subseteq F^o(\Delta_X)$, because
    $F(\Delta_X^o)^o=F(\Delta_X)^o$ and $\Delta_Y\subseteq
    F(\Delta_X)\iff\Delta_Y\subseteq F(\Delta_X)^o$.
  \item $F^o(R_0)\qT F^o(R_1)\subseteq F^o(R_0\qT R_1)$, because
    $\forall R_0,R_1.(R_0\qT R_1)^o=R_1^o\qT R_0^o$ implies
    $F^o(R_0\qT R_1)=F(R_1^o\qT R_0^o)^o$ and $F^o(R_0)\qT
    F^o(R_1)=F(R_0^o)^o\qT F(R_1^o)^o=(F(R_1^o)\qT F(R_o^o))^o$,
    thus\\ $(F(R_1^o)\qT F(R_o^o))^o\subseteq F(R_1^o\qT R_0^o)^o\iff
    F(R_1^o)\qT F(R_0^o)\subseteq F(R_1^o\qT R_0^o)$ and the RHS
    follows from $\forall R_0,R_1.F(R_0)\qT F(R_1)\subseteq F(R_0\qT
    R_1)$.
  \end{itemize}
  In the case of $F$ being strict-monoidal, the proof is obtained by
  replacing $\subseteq$ with $=$ in the last two items.
\end{proof}

\section{Quasi-Uniform Spaces}
\label{sec:q_unif_spaces}

We recall the definition of the category $\QUnif$ of quasi-uniform
spaces and quasi-uniformly continuous maps (for more details see
\parencite{Fletcher_Lindgren:Quasi_Uniform:Book:1982}~and~\parencite[Chapter~E]{Hart_et_al:Encyclopedia_Topology:2003}).
Moreover, we prove that the category of topological
spaces and continuous maps is a coreflective
subcategory of the category of quasi-uniform spaces and
quasi-uniformly continuous maps.
The definitions exploit quantales of binary relations and properties
of lax-monoidal maps between them (see Section~\ref{sec:BR}).

\begin{definition}[Quasi-Uniform Space]
  \label{def:quasi_unif}
  A quasi-uniform space is a pair $(X,\QU)$ consisting of a set $X$
  and a subset $\QU\subseteq\PS(X^2)$ (called a
  \emph{quasi-uniformity}) of binary relations (called
  \emph{entourages}) with the properties:
  \begin{enumerate}
  \item $\QU$ is a filter in the complete lattice $(\PS(X^2),\subseteq)$, \ie
    \begin{itemize}
    \item $\forall R_1 \in \QU. \forall R_2\in\PS(X^2).\ R_1\subseteq
    R_2\implies R_2\in\QU$, \ie, $\QU$ is an upperset;
  \item $X^2\in\QU$ and $\forall R_1,R_2\in\QU.\ R_1\cap R_2\in\QU$,
    \ie, $\QU$ is closed under finite intersections.
    \end{itemize}
  \item Identity property, \ie, $\forall R\in\QU.\Delta\subseteq R$,
    where $\Delta$ is the diagonal on $X$.
  \item Composition property, \ie, $\forall R\in\QU.\exists
    R_0\in\QU.R_0^2\subseteq R$.
  \end{enumerate}
  Given two quasi-uniform spaces $(X,\QU)$ and $(X',\QU')$, a map
  $f:X\to X'$ is called \emph{quasi-uniformly continuous} from
  $(X,\QU)$ to $(X',\QU')$ when $\forall
  R'\in\QU'.f^\dagger(R')\in\QU$, where $f^\dagger$ is defined in
  Proposition~\ref{prop:BR:dagger}.

  We let $\QUnif$ denote the category of quasi-uniform spaces and
  quasi-uniformly continuous maps.
\end{definition}

By analogy with topologies, one can define a notion of base and
sub-base for quasi-uniformities.

\begin{definition}[Base and Sub-base]
  \label{def:QUnif_Base_Subbase}
  Given a quasi-uniformity $\QU$ on $X$ and a subset
  $B\subseteq\QU$, we say that:
  \begin{itemize}
  \item $B$ is a \emph{base} for $\QU$ $\defiff$ $\QU$ is the upperset in
    $\PS(X^2)$ generated by $B$, \ie, $\forall R\in\QU.\exists
    R'\in B.R'\subseteq R$.
  \item $B$ is a \emph{sub-base} for $\QU$ $\defiff$ $\QU$ is the
    filter in $\PS(X^2)$ generated by $B$, or equivalently, the
    closure of $B$ under finite intersections is a base for $\QU$.
  \end{itemize}
\end{definition}

The biggest quasi-uniformity on $X$ is the principal filter
$\upperSetOf{\Delta}$ in $\PS(X^2)$ generated by the diagonal $\Delta$
on $X$, while the smallest quasi-uniformity is the principal filter
$\upperSetOf{X^2} = \setf{X^2}$.

The following result gives an intrinsic characterization of bases for
quasi-uniformities on $X$.

\begin{proposition}[{\parencite[pp.~2]{Fletcher_Lindgren:Quasi_Uniform:Book:1982}}]
  \label{prop:qu:base}
  A non-empty subset $B\subseteq\PS(X^2)$ is a base for a
  quasi-uniformity on $X$ $\iff$
  \begin{equation*}
    (\forall R\in B.\Delta\subseteq R)\land
  (\forall R_0,R_1\in B.\exists R\in B.R\subseteq R_0\cap R_1)\land
  (\forall R\in B.\exists R_0\in B.R_0^2\subseteq R).
  \end{equation*}
\end{proposition}

In the rest of this section we relate $\QUnif$ to other categories
($\Top$, $\Po$ and $\Set$) by extending Diagram~\eqref{diag:forget1}.
Furthermore, we define an involution on $\QUnif$, which
extends/refines that on $\Po$. In Diagram~\eqref{diag:forget2} below:

\begin{itemize}
\item The functors $U:\Top\to\Set$, $U:\Po\to\Set$, $P:\Top\to\Po$ and
  $L_P:\Po\to\Top$ are those from Diagram~\eqref{diag:forget1}.
\item $U:\QUnif\to\Set$ is the strict and small topological functor defined in
    Theorem~\ref{thm:qu:TF}.
\item $T:\QUnif\to\Top$ is the faithful functor defined in
  Proposition~\ref{prop:qu}.
\item $L_T:\Top\to\QUnif$ is the $T$-section, which is left-adjoint to $T$ (see
  Theorem~\ref{thm:las:T})
\item $-^o:\QUnif\to\QUnif$ is the involution defined in
  Proposition~\ref{prop:qu}, and its relations with the involution
  on $\Po$ follow from Proposition~\ref{prop:qu:sp}~and~\ref{prop:aqu}.
\end{itemize}

\begin{equation}
  \begin{tikzcd}[row sep = large, column sep = huge]
    \QUnif \arrow[r, "T"'] \arrow[rd, "U"'] &
    \Top \arrow[l,bend right,dashed,"\bot","L_T"'] \arrow[d, "U"] \arrow[r, "P"']&
    \Po \arrow[l,bend right,dashed,"\bot","L_P"'] \arrow[ld, "U"] \\
  &\Set
  \end{tikzcd}
  \quad
  \begin{tikzcd}[row sep = large, column sep = huge]
  \Po\arrow[d, "-^o"'] \arrow[r, "L_T\circ L_P"]&
  \QUnif\arrow[d, "-^o"] \arrow[r, "P\circ T"]&
  \Po\arrow[d, "-^o"]\\
  \Po\arrow[r, "L_T\circ L_P"'] &\QUnif\arrow[r, "P\circ T"'] &\Po
\end{tikzcd}
\label{diag:forget2}
\end{equation}

To prove that the functor $U:\QUnif\to\Set$ is topological, we first
establish some properties of its fibers, where the fiber $\QUnif_X$ is
isomorphic to the poset of quasi-uniformities on $X$ ordered by the
superset relation.

\begin{proposition}\label{prop:quX:CL}
  Let $\QUnif_X$ be the poset of quasi-uniformities on $X$ ordered by
  $\QU\sqsubseteq\QV\defiff\QU\supseteq\QV$, then:
  \begin{enumerate}
  \item $\QUnif_X$ is a complete lattice, and the meet $\qM_{i\in
    I}\QU_i$ is the quasi-uniformity generated by the sub-base
    $\bigcup_{i\in I}\QU_i$.
  \item If $f:\Set(X,X')$, then the map
    $\QUnif_f:\QUnif_{X'}\to\QUnif_X$, sending a quasi-uniformity
    $\QU'$ on $X'$ to the upperset in $\PS(X^2)$ generated by
    $\setb{f^\dagger(R')}{R' \in \QU'}$, is meet preserving.
  \item If $f:\Set(X,X')$, $\QU\in\QUnif_X$, and $\QU'\in\QUnif_{X'}$,
    then
    $f:\QUnif((X,\QU),(X',\QU'))\iff\QU\sqsubseteq\QUnif_f(\QU')$.
  \end{enumerate}
\end{proposition}

\begin{proof}
  We prove that the meet $\qM_{i\in I}\QU_i$
  is the filter $\QU$ in $\PS(X^2)$ generated by
  $B\defeq\bigcup_{i\in I}\QU_i$, \ie, $B$ is a sub-base for $\QU$.
  First, we prove that $\QU\in\QUnif_X$.  By definition, $\QU$ is a
  filter.  The identity property $\forall R\in\QU.\Delta\subseteq R$
  hold, because $\forall R\in B.\Delta\subseteq R$.
  For the composition property, consider $R'\in\QU$:
  \begin{itemize}
  \item by definition of $\QU$, we have $\bigcap_{i\in
    I_0}R'_i\subseteq R'$ for some $I_0\subseteq_f I$ and
    $I_0$-indexed family $(R'_i\in\QU_i|i \in I_0)$;
  \item by the composition property of $\QU_i$, for each $i\in I_0$
    there exists an $R_i\in\QU_i$ such that $R_i^2\subseteq R'_i$;
    
  \item let $R\defeq\bigcap_{i\in I_0}R_i\in\QU$, then
    $R^2\subseteq\bigcap_{i\in I_0}R_i^2\subseteq\bigcap_{i\in
    I_0}R'_i\subseteq R'$.
  \end{itemize}
  Finally, to prove that $\QU$ is the meet $\qM_{i\in I}\QU_i$, it
  suffices to observe that for every $\QU'\in\QUnif_X$ we have
  the following chain of equivalences:
  $\QU\subseteq\QU'\iff B\subseteq\QU'\iff\forall i\in
  I.\QU_i\subseteq\QU'$.
  
  The map $\QUnif_f$ is well-defined and monotone, because $f^\dagger$
  is lax-monoidal (see Proposition~\ref{prop:BR:dagger}).  Therefore,
  if $B'$ is a base for a quasi-uniformity on $X'$ (\ie, $B'$
  satisfies the properties required in
  Proposition~\ref{prop:qu:base}), then
  $B \defeq \setb{f^\dagger(R')}{R' \in B'}$ is a base for a quasi-uniformity
  on $X$.
  Finally, $\QUnif_f$ is meet-preserving, because:
  \begin{itemize}
  \item If $B'_i$ is a sub-base for $\QU'_i$, then
    $\bigcup_i B'_i$ is a sub-base for $\qM_i \QU'_i$.
  \item If $B_i'$ is a sub-base for $\QU'_i$, then
    $\setb{f^\dagger(R)}{R\in B'_i}$ is a sub-base for $\QUnif_f(\QU'_i)$,
    since $f^\dagger$ preserves intersections (and unions) of relations.
  \end{itemize}
  The third claim follows immediately from the definition of arrows in
  $\QUnif$ (see Definition~\ref{def:quasi_unif}).
\end{proof}

\begin{remark}
  \label{rmk:quX:CL}
  There are alternative descriptions of $\qM_{i\in I}\QU_i$ in special
  cases.  If the family $\QU_i$ is filtered in $\QUnif_X$ (i.e.,
  directed with respect to subset inclusion), then its meet is
  $\cup_{i\in I}\QU_i$. For a finite family indexed over (say)
  $I \in \omega$, if $B_i$ is a base for $\QU_i$ for each $i\in I$,
  then $B\defeq\setb{\cap_{i\in I}R_i}{\forall i\in I.R_i\in B_i}$ is
  a base for $\qM_{i\in I}\QU_i$.
\end{remark}

\begin{theorem}\label{thm:qu:TF}
  Let $U:\QUnif\to\Set$ be the functor such that $U(X,\QU)\defeq X$
  and on arrows is hom-sets inclusion, then $U$ is a strict and small
  topological functor.
\end{theorem}

\begin{proof}
  First, we prove that $U$ is topological (see
  Definition~\ref{def:TF}).
  Given a family $\famb{(Y_i,\QV_i)}{i\in I}$ of quasi-uniform spaces
  and a family $\famb{f_i:X\to Y_i}{i\in I}$ of arrows in $\Set$,
  consider the following quasi-uniformities
  $\QU_i\defeq\QUnif_{f_i}(\QV_i)$ and $\QU\defeq\qM_{i\in I}\QU_i$ on
  $X$.  We show that $\famb{f_i:(X,\QU)\to (Y_i,\QV_i)}{i\in I}$ is an
  initial structure for the triple $(X,\famb{(Y_i,\QV_i)}{i\in
    I},\famb{f_i:X\to Y_i}{i\in I})$.
  By definition of meet in $\QUnif_X$, we have
  $\forall i\in I.\QUnif_{f_i}(\QV_i)\subseteq\QU$.  Therefore, by
  Proposition~\ref{prop:quX:CL},
  $\famb{f_i:(X,\QU)\to (Y_i,\QV_i)}{i\in I}$ is an $I$-cone from
  $(X,\QU)$ in $\QUnif$.

  Given any $I$-cone $\famb{g_i:(X',\QU')\to (Y_i,\QV_i)}{i\in I}$
  from $(X',\QU')$ in $\QUnif$ (\ie, $\forall i\in
  I.\QUnif_{g_i}(\QV_i)\subseteq\QU'$) such that $\forall i\in
  I.g_i=f_i\circ f$ for some $f:\Set(X',X)$, we prove that
  $f:\QUnif( (X',\QU'), (X,\QU))$ (\ie,
  $\QUnif_f(\QU)\subseteq\QU'$). Uniqueness follows from $U$ being
  faithful.
  By Proposition~\ref{prop:quX:CL}, $\QUnif_f$ preserves meets,
  therefore $$\QUnif_f(\QU)=\qM_{i\in I}\QUnif_f(\QUnif_{f_i}(\QV_i))=
  \qM_{i\in I}\QUnif_{f_i\circ f}(\QV_i)= \qM_{i\in
    I}\QUnif_{g_i}(\QV_i)\subseteq\QU'.$$
  Finally, as a topological functor $U$ is strict and small, because
  the fibers $\QUnif_X$ are posets.
\end{proof}

We define the functor $T:\QUnif\to\Top$ and the involution $-^o$ on
$\QUnif$, and then we prove the properties stated in
Diagram~\eqref{diag:forget2}.

\begin{proposition}
  \label{prop:qu}
  If $(X,\QU)$ is a quasi-uniform space and $B$ is a base for
  $\QU$:
  \begin{itemize}
  \item Let $\tau_\QU \defeq \setb{O\subseteq X}{\forall x\in
    O.\exists R\in B.R(x)\subseteq O}$, then $\tau_\QU$ is a topology
    on $X$ independent of the choice of the base $B$ for $\QU$.
    Furthermore, let $T(X,\QU) \defeq (X,\tau_\QU)$, then $T$ extends
    to a functor $T:\QUnif\to\Top$, which on arrows is hom-sets
    inclusion.
    
\item Let $\QU^o \defeq \setb{R^o}{R\in\QU}$ where $R^o$ is the dual
  of $R$. Then, $\QU^o$ is a quasi-uniformity on $X$ with base $B^o$.
  Let the \emph{dual quasi-uniform space} of $(X,\QU)$ be
  $(X,\QU)^o \defeq (X,\QU^o)$. Then, $-^o$ extends to an involution
  on $\QUnif$, which is the identity on hom-sets.
  \end{itemize}
\end{proposition}

\begin{proof}
  First, we prove that $\tau_\QU$ is a topology on $X$, \ie, it is
  closed under arbitrary unions and finite intersections:
  \begin{itemize}
  \item if $A\subseteq\tau_\QU$ and $O'=\cup A$, then $x\in O'\implies
    x\in O$ for some $O\in A\subseteq \tau_\QU$, thus $\exists
    R\in\QU.R(x)\subseteq O\subseteq O'$.
  \item if $A\subseteq_f\tau_\QU$ and $O'=\cap A$, then
    $x\in O'\implies\forall O\in A.x\in O$. Thus,
    $\forall O\in A.\exists R_O\in\QU.R_O(x)\subseteq O$. Let
    $R=\cap\setb{R_O}{O\in A}$, we have $R\in\QU$ (because $\QU$ is a
    filter) and $R(x)=\cap\setb{R_O(x)}{O\in A}\subseteq O'$.
  \end{itemize}
  We prove that $f:\QUnif((X',\QU'),(X,\QU))$ implies
  $f:\Top((X',\tau_{\QU'}),(X,\tau_{\QU}))$, \ie,
  $O\in\tau_{\QU}$ implies $O'\defeq f^{-1}(O)\in\tau_{\QU'}$:
  \begin{itemize}
  \item $x\in O'\implies f(x)\in O\implies
    R(f(x))\subseteq O$ for some $R\in\QU$, therefore
  \item $R'\defeq f^\dagger(R)\in\QU'$ (because $f$ is quasi-uniformly
    continuous) and
  \item $R'(x)=\setb{x'}{f(x')\in R(f(x))}\subseteq \setb{x'}{f(x')\in
    O}=O'$.
  \end{itemize}

  Next, we prove that $\QU^o$ is a quasi-uniformity. By
  Proposition~\ref{prop:BR}, we have that $-^o:\PS(X^2)\to\PS(X^2)$ is
  a strict-monoidal isomorphism, which preserves intersections. Thus,
  it maps the filter $\QU$ to a filter $\QU^o$, while the other
  properties of a quasi-uniformity are preserved because $-^o$ is
  strict-monoidal. More precisely, we have:
  $\Delta\subseteq R\implies \Delta=\Delta^o\subseteq R^o$ and
  $R_0\qT R_0\subseteq R\implies R_0^o\qT R_0^o=R_0^o\qT^o
  R_0^o\subseteq R^o$.

  Finally, $f:\QUnif((X_1,\QU_1),(X_2,\QU_2))$ implies
  $f:\QUnif((X_1,\QU_1^o),(X_2,\QU_2^o))$, since
  $f^\dagger(R^o)=f^\dagger(R)^o$.
\end{proof}

\begin{remark}
  \label{rmk:admit}
  The objects $(X,\QU)$ of $\QUnif$ such that $(X,\QU) =(X,\QU)^o$ are
  called \emph{uniform spaces}. Such spaces are too restrictive for
  capturing all topological spaces. In fact, a topology $\tau$ on $X$
  \emph{admits} a uniformity $\QU$ on $X$ (\ie, $(X,\tau)=T(X,\QU)$)
  if and only if $(X,\tau)$ is completely regular
 ~\parencite[Theorem~38.2]{Willard:General_Topology:Book:1970}.
  On the other hand, every topology $\tau$ on $X$ admits a
  quasi-uniformity~\parencite[pp.~27]{Fletcher_Lindgren:Quasi_Uniform:Book:1982},
  \eg, the Cs{\'a}sz{\'a}r-Pervin quasi-uniformity $\QU^C_\tau$ on $X$
  generated by the sub-base $B^C_\tau=\setb{R_O}{O\in\tau}$, where
  $R_O\defeq\setb{(x,y)\in X^2}{x\in O\implies y\in O}$.
  The functor $C:\Top\rTo\QUnif$, which maps $(X,\tau)$ to
  $(X,\QU^C_\tau)$ and on arrows is hom-sets inclusion, makes $\Top$ a
  full sub-category of $\QUnif$, because $R_{f^{-1}(O)}=f^\dagger(R_O)$
  for every map $f:X'\to X$ and $O\in\tau$.
  The Cs{\'a}sz{\'a}r-Pervin quasi-uniformity is \emph{transitive},
  \ie, it has a sub-base consisting of transitive relations
  $R^2\subseteq R$.  Indeed, all relations in $B^C_\tau$ are
  transitive.
\end{remark}

A quasi-uniformity $\QU$ on $X$ induces two topologies on $X$
($\tau_\QU$ and $\tau_{\QU^o}$) and, accordingly, four product
topologies on $X^2$. If $\QU$ is a uniformity (\ie, $\QU=\QU^o$), then
it induces only one topology on $X$ and one product topology on
$X^2$. Since the elements in a quasi-uniformity $\QU$ are subsets of
$X^2$, one may wonder whether $\QU$ has a base consisting only of open
subsets for one of the four product topologies on $X^2$.

\begin{definition}[Open Entourage and Base]\label{def:open-base}
Given $(X,\QU)\in\QUnif$, we say that an entourage
$R\in\QU$ is \emph{open} $\defiff$ it is open in $T(X,\QU^o)\times
T(X,\QU) = (X^2, \tau_{\QU^o}\times\tau_\QU)$, where
$\tau_1\times\tau_2$ is the product topology on $X_1\times X_2$
induced by the topologies $\tau_i$ on $X_i$.  A
(sub-)base $B$ for $\QU$ is \emph{open} if it contains only open
entourages of $\QU$.
\end{definition}

The following lemma gives an alternative characterization of the
interior operator for the topologies $\tau_\QU$ and
$\tau_{\QU^o}\times\tau_\QU$, and show that a (sub-)base for a
quasi-uniformity can always be \emph{refined} into an open one.

\begin{lemma}\label{lem:interior}
  Given a quasi-uniform space $(X,\QU)$ and a base $B$ for $\QU$, the
  following hold:
  \begin{enumerate}
  \item \label{item:interior_tau_U} If $A\subseteq X$, then $\Int[\tau_\QU]{A} = A'\defeq\setb{x\in
    X}{\exists R\in B.\ R(x)\subseteq A}$.
  \item \label{item:interior_tau_U_dual_U} If $S\subseteq X^2$, then $\Int[\tau_{\QU^o}\times\tau_\QU]{S} =
    S'\defeq\setb{(x_1,x_2)\in X^2}{\exists R\in B.\ R^o(x_1)\times
      R(x_2)\subseteq S}$.
  \end{enumerate}
\end{lemma}

\begin{proof}
  Item~\ref{item:interior_tau_U}: If $O\in\tau_\QU$ and
  $O\subseteq A$, then $O\subseteq A'$.  Hence, it suffices to prove
  that $A'\in\tau_\QU$.  Let $x\in A'$, then $R(x)\subseteq A$ for
  some $R\in\QU$.  By the composition property, there is $R_0\in\QU$
  such that $R_0^2\subseteq R$.  We show that $R_0(x)\subseteq A'$.
  For all $y\in R_0(x)$ (\ie, $(x,y)\in R_0$) we have
  $R_0(y) \subseteq R_0^2(x) \subseteq R(x) \subseteq A$, hence
  $y \in A'$.
  
  Item~\ref{item:interior_tau_U_dual_U}: By definition of product
  topology, $(x_1,x_2)\in\Int[\tau_{\QU^o}\times\tau_\QU]{S}$ implies
  $(x_1,x_2)\in O_1\times O_2 \subseteq S$ for some
  $O_1\in\tau_{\QU^o}$ and $O_2\in\tau_\QU$.
  Thus, there are $R_1, R_2\in \QU$ such that $R_1^o(x_1)\subseteq
  O_1$ and $R_2(x_2)\subseteq O_2$, and therefore $R_1^o(x_1)\times
  R_2(x_2)\subseteq S$.
  Since $\QU$ is a filter, $R\defeq R_1\cap R_2\in\QU$ and
  $R^o(x_1)\times R(x_2) \subseteq R_1^o(x_1)\times R_2(x_2) \subseteq
  S$, proving that $\Int{S}\subseteq S'$.
  On the other hand, if $(x_1,x_2)\in S'$, by definition of $S'$ there
  is $R\in \QU$ such that $R^o(x_1)\times R(x_2)\subseteq S$.
  Let $O_1\defeq\Int[\tau_{\QU^o}]{R^o(x_1)}$ and $O_2\defeq
  \Int[\tau_\QU]{R(x_2)}$. Then, by Item~\ref{item:interior_tau_U},
  $x_1\in O_1$ and $x_2 \in O_2$.  Hence, $(x_1,x_2) \in O_1\times
  O_2\subseteq R^o(x_1)\times R(x_2) \subseteq S$, proving that
  $(x_1,x_2) \in \Int{S}$, therefore $S'\subseteq \Int{S}$
\end{proof}

\begin{proposition}\label{prop:open-entourage}
  If $B$ is an open (sub-)base for a quasi-uniformity $\QU$ on $X$,
  then $B_\tau\defeq\setb{R(x)}{R\in B \land x\in X}$ is a (sub-)base
  for the topology $\tau_\QU$.
\end{proposition}

\begin{proof}
  First, we prove that, if $R\in\QU$ is open and $x\in X$, then $R(x)$
  is open in $\tau_\QU$. Since $y\in R(x) \iff (x,y)\in R$ and $R$ is
  open, then $(x,y)\in O_1\times O_2\subseteq R$ for some
  $O_1\in\tau_{\QU^o}$ and $O_2\in\tau_{\QU}$. Thus,
  $y\in O_2=(O_1\times O_2)(x)\subseteq R(x)$. Therefore, $B$ being
  open implies $B_\tau\subseteq\tau_\QU$.
  Secondly, we prove that $B_\tau$ is a (sub-)base for $\tau_\QU$,
  \ie, $x\in O\in\tau_\QU$ implies $x\in\cap_{i\in n}O_i\subseteq O$,
   for some $n$-indexed family of elements in $B_\tau$.
  If $x\in O\in\tau_\QU$, then (by item~\ref{item:interior_tau_U} of
  Lemma~\ref{lem:interior}) $x\in R(x)\subseteq O$ for some $R\in\QU$.
  Since $B$ is a (sub-)base for $\QU$, $\cap_{i\in n}R_i\subseteq R$,
  for some $n$-indexed family of entourages in $B$.  Therefore,
  $x\in\cap_{i\in n}O_i=(\cap_{i\in n}R_i)(x)\subseteq R(x)\subseteq
  O$, where $O_i\defeq R_i(x)\in B_\tau$.
\end{proof}

The following result is a generalization of an analogous one for
uniform spaces~\parencite[Theorem 6, pp. 179]{Kelley1975}.  It shows that
every quasi-uniformity $\QU$ is closed under the interior operator for
the topology $\tau_{\QU^o}\times\tau_\QU$.
\begin{proposition}
  If $\QU$ is a quasi-uniformity and $R\in\QU$, then
$\Int[\tau_{\QU^o}\times\tau_\QU]{R}\in\QU$.
\end{proposition}

\begin{proof}
  We must find $R_0\in\QU$ such that $R_0\subseteq\Int{R}$.
  By composition and identity properties, there is $R_0\in\QU$
  such that $R_0 \subseteq R_0^3\subseteq R$. We show that $R_0^3 =
  \bigcup_{(x,y)\in R_0} R_0^o(x)\times R_0(y)$, by the following
  chain of equivalences: $(x',y') \in R_0^3 \iff \exists (x,y)\in
  R_0. (x',x)\in R_0 \land (y,y')\in R_0 \iff \exists (x,y)\in
  R_0. x'\in R_0^o(x) \land y'\in R_0(y) \iff (x',y')\in
  \bigcup_{(x,y)\in R_0} R_0^o(x)\times R_0(y)$.  Therefore, by
  Item~\ref{item:interior_tau_U_dual_U} of Lemma~\ref{lem:interior},
  we get $R_0 \subseteq R_0^3 \subseteq \Int{R}$.
\end{proof}

The previous result implies that any (sub-)base for a quasi-uniformity
can be refined into an open one without increasing its cardinality.
\begin{corollary}
  \label{cor:open-base}
  If $B$ is a (sub-)base for the quasi-uniformity $\QU$, then so is
  $B' \defeq \setb{\Int[\tau_{\QU^o}\times\tau_\QU]{R}}{R\in B}$.
\end{corollary}

Every $T$-section $F:\Top\to\QUnif$ gives a different way to view $\Top$
as a full sub-category of $\QUnif$.  Among such $T$-sections there is a
\emph{best} choice, namely the $T$-section which is also left-adjoint to
$T$, making $\Top$ a coreflective sub-category of $\QUnif$.

\begin{theorem}\label{thm:las:T}
  The functor $T:\QUnif\to\Top$ has a left-adjoint $T$-section
  $L_T:\Top\to\QUnif$ such that $L_T(X,\tau)\defeq(X,\QU_\tau)$, where
  $\QU_\tau\defeq\qM\setb{\QV\in\QUnif_X}{\tau_\QV\subseteq \tau}$,
  and on arrows is hom-sets inclusion.
\end{theorem}

\begin{proof}
  First, we prove that $L_T$ is a $T$-section, \ie,
  $T(X,\QU_\tau)=(X,\tau)$ or equivalently $\tau=\tau_{\QU_\tau}$.
  The inclusion $\tau\subseteq\tau_{\QU_\tau}$ holds, because there
  exists $\QV \in \QUnif_X$ such that $\tau_\QV = \tau$, \eg, the
  Cs{\'a}sz{\'a}r-Pervin quasi-uniformity (see
  Remark~\ref{rmk:admit}), and $\tau_\QV\subseteq\tau_\QU$ when
  $\QV\subseteq\QU$ in $\QUnif_X$.
  To prove the inclusion $\tau\supseteq\tau_{\QU_\tau}$, choose an
  open base $B_\QV$ for each quasi-uniformity $\QV\in\QUnif_X$ such
  that $\tau_\QV\subseteq\tau$.
  By the definition of join in $\QUnif_X$ (see
  Proposition~\ref{prop:quX:CL}), $B\defeq\bigcup
  \setb{B_\QV}{\QV\in\QUnif_X\land\tau_\QV\subseteq\tau}$ is a
  sub-base for the quasi-uniformity $\QU_\tau$.  Let $B'$ be the base
  for $\QU_\tau$ generated by $B$, \ie, the closure of $B$ under
  finite intersections.
  For every $O \in \tau_{\QU_\tau}$, by Lemma~\ref{lem:interior}, we
  have $\forall x \in O.\exists R \in B'. R(x)\subseteq O$, but
  $R=\cap_{i\in n}R_i$ for some $\setb{R_i}{i\in n}\subseteq_f B$, \ie,
  $R_i\in B_{\QV_i}$ for some $\QV_i\in\QUnif_X$ such that
  $\tau_{\QV_i}\subseteq \tau$.
  Since $B_{\QV_i}$ is an open base for $\QV_i$ for each $i\in n$, by
  Proposition~\ref{prop:open-entourage}, we get $\forall i\in
  n.R_i(x)\in\tau_{\QV_i}\subseteq\tau$ for each $x\in X$.  Therefore,
  $R(x)=\cap_{i\in n}R_i(x)\in\tau$, since $\tau$ is closed under
  finite intersections.
  Finally, $\forall x \in O.\exists R \in B'. R(x)\subseteq O$
  implies $O\in\tau$, as $x\in R(x)\in\tau$ for every $R\in B'$ and $x
  \in X$.
  
  Secondly, we prove that $L_T$ is left-adjoint to $T$, \ie,
  $f:\Top((X,\tau),(Y,\tau_\QV)) \iff f:\QUnif((X,\QU_\tau),(Y,\QV))$
  when $(X,\tau):\Top$ and $(Y,\QV): \QUnif$.  The ($\impliedby$)
  direction follows from $T(X,\QU_\tau)=(X,\tau)$ and $T$ being
  hom-sets inclusion on arrows. For the ($\implies$) direction, by
  Proposition~\ref{prop:quX:CL}, we have to prove that $\QV_f\defeq
  \QUnif_f(\QV)\subseteq\QU_\tau$, which (by definition of $\QU_\tau$)
  amounts to proving that $\tau_{\QV_f}\subseteq\tau$.
  Consider an open base $B_\QV$ of $\QV$.  Then,
  $B_f=\setb{f^\dagger(R)}{R\in B_\QV}$ is a base for $\QV_f$.
  By Lemma~\ref{lem:interior}, $O\in\tau_{\QV_f}$ implies $\forall x
  \in O. \exists R'\in B_f. R'(x) \subseteq O$, which is equivalent to
  $\forall x \in O. \exists R\in B_\QV. f^\dagger(R)(x)\subseteq O$.
  Since $f^\dagger(R)(x) = \setb{x'\in X}{(f(x),f(x'))\in R} =
  f^{-1}(R(f(x))$ and $R(f(x))\in\tau_\QV$ (by
  Proposition~\ref{prop:open-entourage}), continuity of $f$ implies
  that $f^\dagger(R)(x)=f^{-1}(R(f(x))\in\tau$.
  Hence, we conclude that $O\in\tau$, proving that $\tau_{\QV_f}\subseteq\tau$.
\end{proof}

By composing the left-adjoints $L_P:\Po\to\Top$ and
$L_T:\Top\to\QUnif$ we get a $(P\circ T)$-section $L_T\circ L_P$
left-adjoint to the forgetful functor $P\circ T:\QUnif\to\Po$.
We conclude this section by giving direct descriptions of the
specialization preorder $P(T(X,\QU))$ and the
\emph{Alexandroff quasi-uniform space} $L_T(L_P(X,\leq))$.

\begin{proposition}\label{prop:qu:sp}
  If $(X,\QU)\in\QUnif$, then $P(T(X,\QU))=(X,\bigcap\QU)$,
  \ie, $\leq_\QU\defeq\leq_{\tau_\QU}=\bigcap\QU$.
\end{proposition}

\begin{proof}
  By Corollary~\ref{cor:open-base}, $\QU$ has an
  open base $B$ and $O\in\tau_\QU\iff\forall x\in O.\exists R\in
  B.R(x)\subseteq O$.
  Therefore, since $B$ is open, we have that $B_\tau\defeq\setb{R(z)}{z\in
    X\land R\in B}$ is a base for $\tau_\QU$.
  By definition, the specialization preorder for $\tau_\QU$ is $x\leq
  y\iff\forall O\in B'.x\in O \implies y \in O$, where $B'$ can be
  any sub-base for $\tau_U$.
  If $B'$ is replaced by $B_\tau$, then $x\leq y \iff\forall z\in
  X.\forall R \in B. x\in R(z) \implies y \in R(z)$.
  We now prove that $x\leq y \iff \forall R \in B. y \in R(x)$:
  \begin{itemize}
  \item $\Longrightarrow$ holds, since $\forall R\in B.\forall x\in
    X.x\in R(x)$, by the identity property;
  \item $\Longleftarrow$ holds, since
    $\forall R\in B.\forall z,x\in X.  x\in R(z)\implies\exists R_0\in
    B.R_0(x)\subseteq R(z) \implies \exists R_0 \in B. y \in R_0(x)
    \subseteq R(z) \implies y \in R(z)$.
  \end{itemize}
  Finally, we have $\forall R \in B. y \in
  R(x)\iff \forall R \in B. (x,y) \in R \iff (x,y)\in \bigcap B\iff
  (x,y)\in \bigcap \QU$.
\end{proof}

\begin{proposition}\label{prop:aqu}
  If $(X,\leq)\in\Po$, then
  $L_T(L_P(X,\leq))=(X,\upperSetOf{\leq})$. As such,
  $\QU_\leq\defeq\QU_{\tau_\leq}=\upperSetOf{\leq}$, \ie, it is the
  principal filter in $\PS(X^2)$ generated by $\leq$.
\end{proposition}

\begin{proof}
  Given a preorder $(X,\leq)$, let
  $\QU \defeq \upperSetOf{\leq} \in\QUnif_X$. Since
  $(P \circ T)\circ (L_T \circ L_P)$ is the identity functor on $\Po$,
  we have $\bigcap(\QU_\leq) = \leq$, and therefore
  $\QU_\leq\subseteq\QU$.
To prove the other inclusion $\QU\subseteq\QU_\leq\defeq
\QU_{\tau_\leq}=\qM\setb{\QV\in\QUnif_X}{\tau_\QV\subseteq\tau_\leq}$,
it suffices to prove that $\tau_\QU \subseteq \tau_\leq$.  Indeed, if
$O\in\tau_\QU$, then $\forall x\in O.\exists R\in\QU.R(x)\subseteq O$.
Since $R\in\QU\iff\leq\subseteq R$, we get the following chain of
implications $x\leq y\implies y\in R(x) \implies y \in O$, \ie, $O$ is
upward closed, which is what is needed to conclude that
$O\in\tau_\leq$.
\end{proof}

\section{Metrization Results}
\label{sec:metrization}

The following diagram depicts some of the functors from
Diagram~\eqref{diag:forget2} and two new functors, $\Phi_u$ and $\Phi_c$
(Definition~\ref{def:Phi}), relating categories of quantale-valued
metric spaces to quasi-uniform and topological spaces.
All functors in the diagram are faithful (as their action on arrows is
hom-sets inclusion) and surjective on objects (for $\Phi_u$ and
$\Phi_c$, surjectivity is proven in the metrization
Theorems~\ref{thm:met:qu} and \ref{thm:met:top}).
\begin{equation}
  \label{diag:Met_QU_Top}
    \begin{tikzcd}[row sep = large, column sep = large]
      \MS_u\arrow[r, "-^o"] \arrow[d, "\Phi_u"']&
      \MS_u \arrow[r, hookrightarrow] \arrow[d, rightarrow, "\Phi_u"] &
      \MS_c \arrow[d, rightarrow, "\Phi_c"]\\
      \QUnif\arrow[r, "-^o"']&
      \QUnif \arrow[r, "T"'] \arrow[rd, "U"'] &
      \Top \arrow[d, "U"] \arrow[r, "P"']&
      \Po \arrow[ld, "U"] \\
  &&\Set
  \end{tikzcd}
  \end{equation}
  
\begin{remark}\label{rem:ms-equiv} 
  The fact that $\Phi_u$ and $\Phi_c$ are full
  (Proposition~\ref{prop:Phi_u_c_full_faithful}) together with the
  metrization results (Theorems~\ref{thm:met:qu}
  and~\ref{thm:met:top}) imply that $\Phi_u$ and $\Phi_c$ are
  equivalences.
  Therefore, $U\circ\Phi_c:\MS_c\to\Set$ and
  $U\circ\Phi_u:\MS_u\to\Set$ are topological functors, but they are
  neither strict nor small, unlike $U:\Top\to\Set$ and
  $U:\QUnif\to\Set$.
  These results show, in a formal way, that $\QUnif$ and $\Top$ are
  the qualitative counterparts of $\MS_u$ and $\MS_c$, respectively.
\end{remark}

\begin{definition}\label{def:Phi}
  We define the two actions $\Phi_u$ and $\Phi_c$ on quantale-valued
  metric spaces as follows:
  \begin{itemize}
  \item $\Phi_u(X,d,Q) \defeq (X,\QU_d)$, in which $\QU_d$ is the
    quasi-uniformity generated by the base
    $B^u_d=\setb{R_\delta}{\delta\in\qwb_B\qI}$, where $B$ is any base
    for the continuous lattice $Q$ and
    $R_\delta=\setb{(x,y)\in X^2}{\delta\ll d(x,y)}$.
  \item $\Phi_c(X,d,Q) \defeq (X,\tau_d)$, in which $\tau_d$ is the
    topology generated by the base
    $B^c_d=\setb{B(x,\delta)}{x\in X\land\delta\in\qwb_B\qI}$, where
    $B$ is any base for the continuous lattice $Q$ and
    $B(x,\delta)=\setb{y\in X}{\delta\ll d(x,y)}$.
  \end{itemize}
\end{definition}

\begin{proposition}
  The actions $\Phi_u$ and $\Phi_c$ are well-defined maps.
\end{proposition}

\begin{proof}
  First, we prove that the two actions do not depend on the choice of
  the base for the continuous lattice $Q$. In particular, we can
  always choose the biggest base (\ie, $Q$). Consider a base $B$ for
  $Q$.  Then, by the interpolation property
  (Proposition~\ref{prop:int:1}):
  \begin{equation*}
    \forall\delta'\in\qwb\qI.\exists\delta\in\qwb_B\qI.\ \delta'\ll\delta.
  \end{equation*}
  Since $\delta'\qle\delta\ll\qI$ implies $R_{\delta}\subseteq
  R_{\delta'}$ and $\forall x\in X.B(x,\delta)\subseteq B(x,\delta')$,
  we conclude that $B^u_d$ and $Q^u_d$ are bases for the same
  quasi-uniformity on $X$ and $B^c_d$ and $Q^c_d$ are bases for the
  same topology on $X$.

  Next, we prove that $\QU_d$, namely the upperset in $\PS(X^2)$
  generated by $B^u_d$, is a quasi-uniformity on $X$:
  \begin{itemize}
  \item $\QU_d$ is a filter, because $B^u_d$ is a filtered subset of
    $\PS(X^2)$.  In fact, $\qwb_B\qI$ is a directed subset of
    $Q$, in particular $\bot\in\qwb_B\qI$, thus $R_\bot\in B^u_d$, and
    $\forall\delta_1,\delta_2\in\qwb_B\qI.\exists\delta\in\qwb_B\qI.
    \delta_1,\delta_2\qle\delta$, thus $R_\delta\subseteq
    R_{\delta_1}\cap R_{\delta_2}$.
  \item $\forall R\in\QU_d.\Delta\subseteq R$, because
    $\forall\delta\in\qwb_B\qI.\Delta\subseteq R_\delta$, since
    $\delta\ll\qI\qle d(x,x)$.
  \item $\QU_d$ has the composition property, because
    $\forall\delta'\ll\qI.\exists\delta\in\qwb_B\qI.
    \delta'\ll\delta\qT\delta$ by Proposition~\ref{prop:int:2}, namely
    \begin{itemize}
    \item $\delta'\ll\qI=\qI\qT\qI\implies$ (by
      Proposition~\ref{prop:int:2})
    \item $\exists\delta_1\in\qwb_B\qI.\delta'\ll\delta_1\qT\qI\implies$
      (by Proposition~\ref{prop:int:2})
    \item
      $\exists\delta_1,\delta_2\in\qwb_B\qI.\delta'\ll\delta_1\qT\delta_2$
    \item since $\qwb_B\qI$ is directed, $\exists
      \delta\in\qwb_B\qI.\delta_1,\delta_2\qle\delta$, which implies
      $\delta'\ll\delta\qT\delta$.
    \end{itemize}
    Therefore, we have the following chain of implications
    $R_\delta^2(x,z)\implies\exists y.\delta\ll
    d(x,y)\land \delta\ll d(y,z)\implies\\\exists y.\delta\qT\delta\qle d(x,y)\qT
    d(y,z)\qle d(x,z)\implies \delta'\ll d(x,z)\implies
    R_{\delta'}(x,z)$.
  \end{itemize}
  For the proof that $\tau_d$ is a topology on $X$ we refer to
 ~\parencite[Def~4.5 and
  Prop~4.8]{Dagnino_Farjudian_Moggi:Robust_Topology:arXiv:2025}, and
  for the proof that $B^c_d$ is a base, rather than a sub-base for
  $\tau_d$, we use~\parencite[Lemma
  4.7]{Dagnino_Farjudian_Moggi:Robust_Topology:arXiv:2025}.
\end{proof}

\begin{remark}
  The cardinality of the base $B^u_d$ for the quasi-uniformity $\QU_d$
  is bounded by the cardinality of the base $B$ for $Q$, while the
  cardinality of the base $B^c_d$ for the topology $\tau_d$ is bounded
  by the cardinality of $B\times X$.
\end{remark}

\begin{proposition} \label{prop:T-Phiu-Phic}
  $T(\Phi_u(X,d,Q))=\Phi_c(X,d,Q)$ for every $(X,d,Q)\in\MS_u$.
\end{proposition}

\begin{proof}
  The claim amounts to proving that $\tau_d$ and $\tau_{\QU_d}$
  coincide.
  We use an alternative characterization of $\tau_d$, namely
  $O\in\tau_d\iff\forall x\in
  O.\exists\delta\in\qwb_B\qI.B(x,\delta)\subseteq O$.  In fact,
  $\epsilon\ll q=q\qT\qI\implies\exists\delta\in\qwb_B\qI.\epsilon\ll
  q\qT\delta$ by Proposition~\ref{prop:int:2}. Thus,
  $\forall\epsilon\ll\qI.\forall x\in B(y,\epsilon).\exists
  \delta\in\qwb_B\qI.B(x,\delta)\subseteq B(y,\epsilon)$, because
  $\epsilon\ll d(y,x)\implies\exists\delta\in\qwb_B\qI.\epsilon\ll
  d(y,x)\qT\delta\implies\forall z\in B(x,\delta).  \epsilon\ll
  d(y,x)\qT\delta\qle d(y,x)\qT d(x,z)\qle d(y,z)$.

  Given a base $B$ for $Q$ and $O\subseteq X$, one has
  \begin{equation*}
   (\forall x\in O.\exists\delta\in\qwb\qI.B(x,\delta)\subseteq O)\iff
  (\forall x\in O.\exists R\in B^u_b.R(x)\subseteq O), 
  \end{equation*}
  because $R_\delta(x)=B(x,\delta)$.  Therefore,
  $\tau_d=\tau_{\QU_d}$.
\end{proof}

\begin{remark}
Thanks to Proposition~\ref{prop:T-Phiu-Phic}, for a quantale-valued
metric space $(X,d,Q)$, we know that $\tau_{\QU_d} = \tau_d$.
Moreover, it is easy to see that the base $B^u_d$ of the
quasi-uniformity $\QU_d$ is open in the sense of
Definition~\ref{def:open-base}.
\end{remark}

\begin{proposition}
  \label{prop:Phi_u_c_full_faithful}
  The functors $\Phi_u$ and $\Phi_c$ are full\&faithful.
\end{proposition}

\begin{proof}
  The claim for $\Phi_u$ amounts to showing that
  $f:\MS_u((X,d,Q),(X',d',Q'))\iff f:\QUnif((X,\QU_d),(X',\QU_{d'}))$
  for every map $f:X \to X'$.  In fact, we have the following chain of
  equivalences:
  \begin{align*}
    f:\MS_u((X,d,Q),(X',d',Q'))  &\iff  \forall\epsilon\ll'\qI'.\exists\delta\ll\qI.\forall
             x,y\in X.\ \delta\ll d(x,y)\implies\epsilon\ll' d'(f(x),f(y))\\
    & \iff \forall\epsilon\ll'\qI'.\exists\delta\ll\qI.\forall
             x,y\in X.\ R_\delta(x,y)\implies R'_\epsilon(f(x),f(y))\\
    & \iff f:\QUnif((X,\QU_d),(X',\QU_{d'})).
  \end{align*}
  Similarly, the claim for $\Phi_c$ amounts to showing that
  $f:\MS_c((X,d,Q),(X',d',Q'))\iff f:\Top((X,\tau_d),(X',\tau_{d'}))$
  for every map $f:X \to X'$.  In fact, we have the following chain of
  equivalences:
  \begin{align*}
     f:\MS_c((X,d,Q),(X',d',Q')) &\iff \forall x\in X.\epsilon\ll'\qI'.\exists\delta\ll\qI.\forall
             y\in X.\ \delta\ll d(x,y)\implies\epsilon\ll' d'(f(x),f(y))\\
    &\iff \forall x\in X.\forall\epsilon\ll'\qI'.\exists\delta\ll\qI.\forall
             y\in B(x,\delta).\ f(y)\in B'(f(x),\epsilon) \\
    &\iff \forall x\in
            X.\forall\epsilon\ll'\qI'.\exists\delta\ll\qI.\ 
            f(B(x,\delta))\subseteq B'(f(x),\epsilon)\\
    &\iff f:\Top((X,\tau_d),(X',\tau_{d'})). \qedhere
  \end{align*}
\end{proof}

\begin{remark}
  To prove that the functor $\Phi_u:\MS_u\rTo\QUnif$ is an
  equivalence, we show that for every quasi-uniform space $(X,\QU)$
  there exists $(X,d,Q)\in\MS_u$ such that $(X,\QU)=\Phi_u(X,d,Q)$.
  All metric spaces $(X,d,Q)$ that map to $( X, \QU)$ are necessarily
  isomorphic in $\MS_u$, since $\Phi_u$ is full\&faithful, but the
  continuous quantales $Q$ could be non-isomorphic and differ even in
  cardinality.
Moreover, we can always restrict to \emph{affine} continuous
quantales, \ie, quantales where the unit $\qI$ is the top element
$\qt$, since every $(X,d,Q)\in\MS_u$ is isomorphic to
$(X,d_\qI,Q_\qI)$, where $Q_\qI$ is $Q$ restricted to $\qLowerSet\qI$
and $d_\qI:X^2\to Q_\qI$ is $d_\qI(x,y)=d(x,y)\qm\qI$.
\end{remark}

  In the context of computable analysis, quasi-uniform spaces that
  possess a countable base are most relevant.  For such spaces, one
  would like the continuous quantale $Q$ to have a countable basis as
  well, so that one can stay within the computational framework of
  effectively given domains~\parencite{smyth1977effectively}.
  The construction that we propose starts from a base $B$ for $\QU$
  and produces a prime-continuous quantale $Q$ that has a prime-base
  of cardinality not greater than that of $B$.  In particular, if
  $\QU$ has a countable base, then the prime-continuous quantale $Q$
  has a countable prime-base as well, which implies that $Q$ as a
  continuous quantale has a countable base.

\subsection{Abstract Prime-bases and Rounded Lowersets}
\label{subsec:Rounded_Lowersets}

We introduce the notions of abstract prime-base $(B,\prec)$ and
rounded lowerset (in an abstract prime-base), which are variants of
the well-known notions of abstract base and rounded ideal~\parencite[Sec
2.2.6]{AbramskyJung94-DT}. We prove that the poset of rounded
lowersets ordered by inclusion is a prime-continuous lattice
(Theorem~\ref{thm:apb:pcl}).

\begin{definition}\label{def:apb}
  An \emph{abstract prime-base} is a pair $(B,\prec)$ with $\prec$ a
  binary relation on $B$ that is:
  \begin{itemize}
  \item \textbf{transitive:} $\forall x,y,z\in B.\ x\prec y\prec z\implies x\prec z$,
  \item \textbf{dense:}
    $\forall x,z\in B.\ x\prec z\implies\exists y\in B.\ x\prec y\prec z$.
  \end{itemize}
  A \emph{rounded lowerset} is a subset $A\subseteq B$ that is:
  \begin{itemize}
  \item \textbf{rounded:} $\forall x\in A.\exists y\in A.\ x\prec y$,
  \item \textbf{lower:}
    $\forall y\in A. \forall x \in B.\ x\prec y \implies x\in A$.
  \end{itemize}
  $\RD(B,\prec)$ denotes the poset of rounded lowersets in $(B,\prec)$
  ordered by inclusion.
\end{definition}

If $(B,\prec)$ is a preorder (\ie, transitive and reflexive), then it is
an abstract prime-base and the rounded lowersets are the lowersets.
In fact, if $\prec$ is reflexive, then $\prec$ is trivially dense and
every $A\subseteq B$ is rounded.

\begin{remark}
  Construction of prime-continuous lattices from abstract prime-bases
  has appeared in the literature before, albeit through a different
  terminology. In the context of information systems,
  Vickers~\parencite{Vickers:Infosys:1993} introduced the concept of
  \emph{infosys}, which is precisely what we call an abstract
  prime-base. Instead of rounded lowersets, he considered rounded
  uppersets (which he called upper-closed
  sets~\parencite[Definition~2.2]{Vickers:Infosys:1993}). As such, the
  prime-continuous lattice that is constructed
  in~\parencite[Proposition~2.3]{Vickers:Infosys:1993} is isomorphic to
  $\RD(B,\succ)$. In fact, as far back as 1953, Raney had already
  investigated this construction of prime-continuous
  lattices~\parencite{Raney:Subdirect:1953}.

  Our terminology is closer to that used in the context of abstract
  bases and rounded ideal completion~\parencite[Sec
  2.2.6]{AbramskyJung94-DT}. We also present the proof of
  Theorem~\ref{thm:apb:pcl}---even though it is
  essentially~\parencite[Proposition~2.3]{Vickers:Infosys:1993}---with the
  intention of making the content self-contained and easier to follow
  for the reader.
\end{remark}

\begin{example}\label{ex:R+:apb}
  Consider the quantale $\RQ_+=([0,\infty],\qle,+,0)$
  in~\parencite{lawvere1973metric}, which allows viewing ordinary
  metric spaces as small $\RQ_+$-enriched categories, where
  $[0,\infty]$ is the set of non-negative real numbers extended with
  $\infty$ and ordered by $x\qle y\iff x\geq y$. Thus, $0$ and
  $\infty$ are the top and bottom elements, respectively.

The linear order $([0,\infty],\qle)$ is a prime-continuous lattice
with $\qJ S=\inf S$, $x\lll y\iff x>y$, and the set ${\Q_{>0}}$ of
positive rational numbers is a countable prime-base.
It is easy to see that $({\Q_{>0}},\lll)$ is an abstract prime-base
and the poset $\RD({\Q_{>0}},\lll)$ is isomorphic to
$([0,\infty],\qle)$, namely an element $x$ in $[0,\infty]$ corresponds
to the rounded lowerset $B(x)=\setb{q\in {\Q_{>0}}}{q>x}$. In
particular, $B(0)={\Q_{>0}}$ and $B(\infty)=\emptyset$.
  Also, $[0,\infty]$ is a prime-base, but uncountable,
  $({[0,\infty]},\lll)$ is an abstract prime-base, and the poset
  $\RD({[0,\infty]},\lll)$ is isomorphic to $([0,\infty],\qle)$,
  namely, an element $x$ in $[0,\infty]$ corresponds to the rounded
  lowerset $(x,\infty]$.  Note that lowersets of the form $[x,\infty]$
  are not rounded.
\end{example}

\begin{theorem}\label{thm:apb:pcl}
  Given an abstract prime-base $(B,\prec)$, the poset $\RD(B,\prec)$
  is a prime-continuous lattice with join given by union, prime-base
  $\setb{\qLowerSet a}{a\in B}$, and
  $A'\lll A\iff\exists a\in A.A' \subseteq\qLowerSet a$, where
  $\qLowerSet a\defeq\setb{x \in B}{x\prec a}$.
\end{theorem}

\begin{proof}
  Given a subset $S\subseteq\RD(B,\prec)$, let $A'=\cup S\subseteq B$,
  which is a rounded lowerset, because:
  \begin{itemize}
  \item $x\prec y\in A'\implies(\exists A\in S.x\prec y\in A)\implies$
    [since $A$ is a lower set]
    $(\exists A\in S.x\in A)\implies x\in A'$.

  \item $x\in A'\implies(\exists A\in S.x\in A)\implies$ [since $A$ is
    rounded]
    $(\exists A\in S.\exists y\in A.x\prec y)\implies \exists y\in
    A'.x\prec y$.\footnote{Note that, so far, the properties of
      $\prec$ have not been used.}
  \end{itemize}
  We have: $\forall a\in B. \ \qLowerSet a\in\RD(B,\prec)$, because:
  \begin{itemize}
  \item $x\prec y\prec a\implies x\prec a$, by transitivity of
    $\prec$,
  \item $x\prec a\implies \exists y\in B.x\prec y\prec a$, by
    density  of $\prec$.
  \end{itemize}
  Furthermore, $A\in\RD(B,\prec)$ implies
  $A=\cup_{a\in A}\qLowerSet a$.  In fact, by definition of a rounded
  lowerset:

  \begin{itemize}
  \item $x\in A\implies\exists a\in A.x\prec a$, thus
    $x\in\qLowerSet a$ for some $a\in A$, which proves that $A
    \subseteq \cup_{a\in A}\qLowerSet a$,
  \item $x\prec a\in A\implies x\in A$, thus
    $\qLowerSet a\subseteq A$ for every $a\in A$, which proves that
    $A \supseteq \cup_{a\in A}\qLowerSet a$.
  \end{itemize}
  
  Therefore, to prove that $\setb{\qLowerSet a}{a\in B}$ is a
  prime-base, it suffices to prove that $a\in A$ implies
  $\qLowerSet a\lll A$ , \ie,
  \begin{equation*}
  \forall S\subseteq\RD(B,\prec).\ A\subseteq\cup S\implies \exists
  A'\in S.\qLowerSet a\subseteq A'.  
  \end{equation*}
  In fact $a\in A\subseteq\cup S$
  implies $a\in A'$ for some $A'\in S$, thus
  $\qLowerSet a\subseteq A'$, because $A'\in\RD(B,\prec)$.
  Finally, we prove that $A'\lll A\iff\exists a\in A.\ A'\subseteq\qLowerSet a$:
  \begin{itemize}
  \item $A'\lll A=\cup_{a\in A}\qLowerSet a\implies\exists a\in
    A.\ A'\subseteq\qLowerSet a$, by definition of $\lll$,
  \item $\exists a\in A.\ A'\subseteq\qLowerSet a\implies A'\lll A$,
    because $\forall a\in A.\qLowerSet a\lll A$, together with
    Item~\ref{item:order_totally_below} of
    Proposition~\ref{prop:orders_and_approx}. \qedhere
  \end{itemize}
\end{proof}

If $(B,\prec)$ is a preorder, then $\RD(B,\prec)$ is prime-algebraic,
since, for every $a\in B$, $a\in\qLowerSet a$ and thus
$\qLowerSet a\lll\qLowerSet a$. Since every prime-continuous lattice
is completely distributive (see~\parencite[Thm 7.1.3]{AbramskyJung94-DT}),
we conclude that $\RD(B,\prec)$ is a \emph{locale}, \ie, a quantale
with meet $\qm$ as tensor and the top element $\qt$ as unit. There is,
however, another way to turn $\RD(B,\prec)$ into a quantale, given an
associative operation $\qT$ on $B$ satisfying certain properties.

\begin{theorem}\label{thm:apb:pcq}
  If $(B,\prec)$ is an abstract prime-base and $\qT:B^2\to B$ is an
  associative operation which is:
  \begin{itemize}
  \item \emph{roundB:} $\forall x\in B.\exists y\in B.x\prec y$,
  \item \emph{monotone:}
    $x\prec x'\land y\prec y'\implies x\qT y\prec x'\qT y'$,
  \item \emph{denseT:}
    $x\prec x'\implies\exists y\in B.\ x\prec y\qT x'\qT y\prec x'$,
  \item \emph{lowerT:}
    $x\prec x'\implies\forall y\in B.\ (y\qT x \prec x') \wedge (x\qT y\prec x')$,
  \end{itemize}
  then, $\RD(B,\prec)$ is a quantale with tensor
  $A_1\hat{\qT}A_2\defeq\bigcup\setb{\qLowerSet(a_1\qT a_2)}{a_1\in
    A_1\land a_2\in A_2}$ and unit $B$.
\end{theorem}

\begin{proof}
  First, we have $A_1\hat{\qT}A_2\in\RD(B,\prec)$, because it is
  the join of elements in $\RD(B,\prec)$, and $B\in\RD(B,\prec)$,
  because it is rounded by (roundB).  Moreover, we have:
  \begin{itemize}
  \item $(\qJ S_1)\hat{\qT}(\qJ S_2)=\qJ\setb{A_1\hat{\qT} A_2}{A_1\in
    S_1\land A_2\in S_2}$ for every $S_1,S_2\subseteq\RD(B,\prec)$,
    because

    $(\qJ S_1)\hat{\qT}(\qJ S_2)= \cup\setb{\qLowerSet(a_1\qT
      a_2)}{\exists A_1\in S_1.a_1\in A_1\land\exists A_2\in
      S_2.a_2\in A_2}=$\\ $\cup\setb{A_1\hat{\qT} A_2}{A_1\in S_1\land
      A_2\in S}=\qJ\setb{A_1\hat{\qT} A_2}{A_1\in S_1\land A_2\in S_2}$.

  \item $(A_1\hat{\qT}A_2)\hat{\qT}A_3=A_1\hat{\qT}(A_2\hat{\qT}A_3)$
    for every $A_1,A_2,A_3\in\RD(B,\prec)$, because

    $(A_1\hat{\qT}A_2)\hat{\qT}A_3=\cup\setb{\qLowerSet(a
      \qT a_3)}{a\in (A_1\hat{\qT} A_2)\land a_3\in A_3}=$
    by definition of $A_1\hat{\qT} A_2$ and (monotone)\\
    $\cup\setb{\qLowerSet((a_1\qT a_2)\qT a_3)}{a_1\in A_1\land a_2\in
      A_2\land a_3\in A_3}=$
    by $\qT$ associative\\
    $\cup\setb{\qLowerSet(a_1\qT(a_2\qT a_3))}{a_1\in A_1\land a_2\in
      A_2\land a_3\in A_3}=A_1\hat{\qT}(A_2\hat{\qT}A_3)$.

  \item $A\hat{\qT}B=A$ for every $A\in\RD(B,\prec)$, because

    $A\hat{\qT}B=\cup\setb{\qLowerSet (a\qT y)}{a\in A\land y\in
      B}\subseteq$ \big[by ($A$ rounded) $\exists a'\in A.a\prec a'$ and
    (lowerT) $a\qT y\prec a'$ \big]\\
    $\cup\setb{\qLowerSet a'}{a'\in A}=A$

    $A=\cup\setb{\qLowerSet a}{a\in A}\subseteq$ \big[by ($A$ rounded)
    $\exists a',a''\in A.\ a\prec a'\prec a''$, (denseT) $\exists y\in
    B.a\prec y\qT a'\qT y\prec a'$, (lowerT) $y\qT a'\prec a''$,
    (roundB) $\exists y'\in B.y\prec y'$ and (monotone) $y\qT a'\qT
    y\prec a''\qT y'$ \big]\\
    $\cup\setb{\qLowerSet (a''\qT y')}{a''\in A\land y'\in B}=A\hat{\qT}B$.

  \item $B\hat{\qT}A=A$ for every $A\in\RD(B,\prec)$ is proved
    similarly. \qedhere
  \end{itemize}
\end{proof}

\begin{example}\label{ex:R+:apb:qT}
  Consider the two abstract bases $(\Q_{>0},\lll)$ and
  $([0,\infty],\lll)$ in Example~\ref{ex:R+:apb}. Addition is an
  associative operation on $\Q_{>0}$ and $[0,\infty]$. Moreover,
  $0\in[0,\infty]$, thus we can ask if it satisfies the properties
  required in Theorem~\ref{thm:apb:pcq}:
  \begin{itemize}
  \item $+:\Q_{>0}^2\to\Q_{>0}$ satisfies all properties, including roundB
  \item $+:[0,\infty]^2\to[0,\infty]$ satisfies all properties,
    except roundB, which fails when $x=0$.
  \end{itemize}
  If $\qT:B^2\to B$ has a unit $\qI$ and $\forall x\in B.x\prec\qI$,
  then roundB is trivially true (take $y=\qI$). For $+:[0,\infty]^2\to[0,\infty]$ the
  unit is $\qI=0$, but $\forall x\in[0,\infty].x\lll\qI$ fails, and the
  biggest rounded lowerset is $\qLowerSet\qI=(0,\infty]\subset[0,\infty]$.
\end{example}

\begin{remark}
  If $(B,\prec)$ is a preorder, then meet in $\RD(B,\prec)$ is given
  by intersection, but when $(B,\prec)$ is not a preorder, even binary
  meet may fail to be given by intersection. For instance, let
  $B\defeq\big(2 \times \Q \big) \cup \setf{\bot}$ and define
  \begin{equation*}
    x \prec y \defiff (x = \bot \land y \neq \bot)\lor (\exists i \in
    2, q,q'\in\Q.\ x = (i,q)\land y = (i,q')\land q < q').
  \end{equation*}
  If we take $a_i \defeq (i,0)$, then $\qLowerSet{a_0} \cap
  \qLowerSet{a_1} = \setf{\bot} \notin \RD(B,\prec)$, while
  $\qLowerSet{a_0} \qm \qLowerSet{a_1}=\emptyset$.
\end{remark}

\subsection{Metrization for Quasi-Uniformities}
\label{sec:met:qu}

Given a base $B$ for a quasi-uniformity $\QU$ on $X$, we define an
abstract prime-base $(B,\prec)$. Hence, we can apply
Theorem~\ref{thm:apb:pcl} to construct a prime-continuous
lattice. Under the additional assumption that $B$ is closed under
composition $\qT$ of binary relations, we can apply
Theorem~\ref{thm:apb:pcq} and construct a prime-continuous affine
quantale $Q_B$, which is the key for defining a metric space
$(X,d_B,Q_B)\in\MS_u$ such that $(X,\QU)=\Phi_u(X,d_B,Q_B)$.

\begin{definition}\label{def-ll-QU}
  Given a base $B$ for a quasi-uniformity $\QU$, the binary relation
  $\prec$ on $\QU$ is defined by
  \begin{equation*}
    R_1\prec R_2 \defiff \exists R\in B.\ R_1\supseteq R\qT R_2\qT R,
  \end{equation*}
  where $\qT$ denotes composition of binary relations on $X$.
\end{definition}

\begin{proposition}
  \label{prop-ll-QU}
  The relation $\prec$ on the quasi-uniformity $\QU$ is independent of
  the choice of base.
\end{proposition}
\begin{proof}
  Given two bases $B$ and $B'$ for $\QU$, let $\prec$ and $\prec'$
  denote the relations in Definition~\ref{def-ll-QU} defined using $B$
  and $B'$, respectively.
  We prove that $\prec$ is included in $\prec'$ (the other inclusion
  follows by swapping the two bases).  In fact, for any $R\in B$ there
  exists a $R'\in B'$ such that $R\supseteq R'$. Therefore, by
  monotonicity of relational composition of binary relations,
  $R\qT R_2\qT R\supseteq R'\qT R_2\qT R'$, for any $R_2\in\QU$.
\end{proof}

The following properties of $\prec$ imply that, for any base $B$ for
$\QU$, the pair $(B,\prec)$ is an
abstract prime-base.

\begin{theorem}
  \label{thm:prec_properties}
  The binary relation $\prec$ on $\QU$ (Definition~\ref{def-ll-QU})
  has the following properties:
  \begin{enumerate}
  \item\label{thm:prec:1} $\prec$ is included in $\supseteq$, and the
    opposite inclusion holds when $\Delta\in\QU$.

  \item\label{thm:prec:supset_prec_comp}
    ${\supseteq\circ\prec\circ\supseteq} = {\prec}$.

  \item\label{thm:prec:transitive} $\prec$ is transitive, \ie, $\prec\circ\prec$
    is included in $\prec$.

  \item\label{thm:prec:directed} $\QU$ is $\prec$-directed, \ie,
    $\forall n\in\omega. \ \big( (\forall i\in n.R_i\in\QU)\implies\exists
    R\in\QU.\forall i\in n.R_i\prec R \big )$.

  \item\label{thm:prec:interpolation} $\prec$ satisfies the
    interpolation property, \ie,
    $\forall n\in\omega. \big((\forall i\in n.R_i\prec R')\implies\exists
    R\prec R'.\forall i\in n.R_i\prec R \big)$.

  \item\label{thm:prec:monotone} Composition $\qT$ of relations is
    $\prec$-monotone, \ie, $R_0\prec R'_0\land R_1\prec R'_1\implies
    R_0\qT R_1\prec R'_0\qT R'_1$.
  \end{enumerate}
  Moreover, the restrictions of $\prec$ and $\supseteq$ to a base $B$
  for $\QU$ satisfy properties
  (\ref{thm:prec:1}--\ref{thm:prec:directed}), and if $B$ is closed
  under $\qT$, then properties \ref{thm:prec:interpolation} and
  \ref{thm:prec:monotone} are satisfied as well.
\end{theorem}

\begin{proof}
  We prove the properties for the relations $\supseteq$ and $\prec$
  restricted to a generic base $B$.
  \begin{enumerate}
  \item If $R_1\prec R_2$, \ie, $R_1\supseteq R\qT R_2\qT R$ for some
    $R\in\QU$, then $R\qT R_2\qT R\supseteq R_2$, because $\forall
    R\in\QU.R\supseteq\Delta$, which implies that $R_1\supseteq R_2$.
    If $\Delta\in\QU$, then $R_1\supseteq R_2$
    implies $R_1\prec R_2$, since $R_2=\Delta\qT R_2\qT\Delta$.

  \item We prove that $\supseteq\circ\prec\circ\supseteq$
    is included in $\prec$. In fact, if $R_1\supseteq R_2\supseteq
    R\qT R_3\qT R\land R_3\supseteq R_4$ for some $R\in\QU$, then
    $R_1\supseteq R_2\supseteq R\qT R_3\qT R\supseteq R\qT R_4\qT R$,
    by monotonicity of $\qT$ and transitivity of $\supseteq$, thus
    $R_1\prec R_4$.

\item $\prec\circ\prec$ is included in $\prec$, since
  $\prec\circ\prec$ is included in $\supseteq\circ\prec$, by property
  \ref{thm:prec:1} and monotonicity of $\circ$, and
  $\supseteq\circ\prec$ is included in
  $\supseteq\circ\prec\circ\supseteq$, by $\Delta_\QU$ included in
  $\supseteq$ and monotonicity of $\circ$, thus $\prec\circ\prec$ is
  included in $\prec$, by property \ref{thm:prec:supset_prec_comp} and
  transitivity of inclusion.

\item If $(\forall i\in n.R_i\in\QU)$, then $R_n\defeq\bigcap_{i\in
  n}R_i\in\QU$, with $U_0=X^2$ when $n=0$. Take $S\in\QU$
  such that $R_n\supseteq S^4$ and $R\in B$ such that
  $S^2\supseteq R$, then $\forall i\in n.R_i\supseteq
  R_n\supseteq S^4\supseteq S\qT R\qT S$, thus
  $\forall i\in n.R_i\prec R$, by monotonicity of $\qT$ and
  transitivity of $\supseteq$.

\item If $(\forall i\in n.R_i\prec R')$, which means that $(\forall
  i\in n.\exists S_i\in\QU.R_i\supseteq S_i\qT R'\qT S_i)$, then
  $S_n\defeq\bigcap_{i\in n}S_i\in\QU$ and $(\forall i\in
  n.R_i\supseteq S_n\qT R'\qT S_n)$.  Take $S\in B$ such that
  $S_n\supseteq S^2$, then $R\defeq S\qT R'\qT S\in B$, if the base
  $B$ is closed under $\qT$, therefore $(\forall i\in n.R_i\supseteq
  S_n\qT R'\qT S_n\supseteq S\qT R\qT S)$ and $R\supseteq S\qT R'\qT
  S$, which imply $(\forall i\in n.R_i\prec R\prec R')$.

\item If the base $B$ is closed under $\qT$, then the property that $\qT$ is
  $\prec$-monotone is equivalent to the implication $(\exists
  S_0,S_1\in\QU.R_0\supseteq S_0\qT R'_0\qT S_0\land R_1\supseteq
  S_1\qT R'_1\qT S_1)\implies \exists S\in\QU.R_0\qT R_1\supseteq S\qT
  R'_0\qT R'_1\qT S$.
  Take $S\defeq S_0\cap S_1\in\QU$, then $R_0\qT R_1\supseteq S\qT
  R'_0\qT S^2\qT R'_1\qT S\supseteq S\qT R'_0\qT R'_1\qT S$, since
  $S\supseteq\Delta$. \qedhere
  \end{enumerate}
\end{proof}

\begin{corollary}\label{corr:Q:B}
  If $B$ is a base closed under $\qT$ for a quasi-uniformity $\QU$,
  then $Q_B\defeq(\RD(B,\prec),\hat{\qT},B)$ is an affine
  prime-continuous quantale.
\end{corollary}
\begin{proof}
  It suffices to check that $(B,\prec)$ is an abstract-prime base and
  $\qT:B^2\to B$ satisfies the properties required in
  Theorem~\ref{thm:apb:pcq}.
  Theorem~\ref{thm:prec_properties} states that $\prec$ is transitive
  and satisfies the interpolation property, which implies density.
  Thus, $(B,\prec)$ is an abstract-prime base.  Moreover, $\qT$
  satisfies the properties required in Theorem~\ref{thm:apb:pcq} for
  the following reasons:
    \begin{itemize}
  \item \emph{roundB:} follows from Theorem~\ref{thm:prec_properties}
    item~\ref{thm:prec:directed}.

  \item \emph{monotone:} corresponds to Theorem~\ref{thm:prec_properties}
    item~\ref{thm:prec:monotone}.

  \item \emph{denseT:} follows from density, the definition of $\prec$ and
    Theorem~\ref{thm:prec_properties} item~\ref{thm:prec:supset_prec_comp}:

    $R_0\prec R_2\implies$ by density $\exists R_1\in B$:\\
    $R_0\prec R_1\prec R_2\implies$ by definition or $\prec$ for some $S\in B$:\\
    $R_0\prec R_1\supseteq S\qT R_2\qT S\implies$ by Theorem~\ref{thm:prec_properties} item~\ref{thm:prec:supset_prec_comp}:\\
    $R_0\prec S\qT R_2\qT S$.

  \item \emph{lowerT:} follows from the definition of $\prec$ and
    Theorem~\ref{thm:prec_properties} item~\ref{thm:prec:supset_prec_comp}:

   $R_0\prec R_1\implies$ for every $S\in B$ (which implies $S\supseteq\Delta$)\\
   $S\qT R_0,R_0\qT S\supseteq R_0\prec R_1\implies$ by Theorem~\ref{thm:prec_properties} item~\ref{thm:prec:supset_prec_comp}\\
   $S\qT R_0,R_0\qT S\prec R_1$. \qedhere
  \end{itemize}
\end{proof}

\begin{theorem}\label{thm:met:qu}
  If $B$ is a base closed under $\qT$ for a quasi-uniformity $\QU$ on
  $X$, let $Q_B$ be the affine prime-continuous quantale defined in
  Corollary~\ref{corr:Q:B} and let $d_B:X^2\to Q_B$ be
  \begin{equation*}
    d_B(x,y)\defeq\bigcup\setb{\qLowerSet R}{(x,y)\in R\in B},
  \text{ where }\qLowerSet R\defeq\setb{R'\in B}{R'\prec R}\in\RD(B,\prec).
  \end{equation*}
  Then, $(X,d_B,Q_B)\in\MS_u$ and $(X,\QU)=\Phi_u(X,d_B,Q_B)$.
\end{theorem}

\begin{proof}
  First, $d_B(x,y)\in Q_B$, because it is the join of a subset of the
  prime-continuous lattice $\RD(B,\prec)$.  Second, we prove that
  $d_B$ is a metric, namely
  \begin{itemize}
  \item $d_B(x,x)=B$, since $\forall R\in B.(x,x)\in\Delta\subseteq R$
    and $B$ is rounded, \ie, $B=\bigcup\setb{\qLowerSet R}{R\in B}$.
  \item $d_B(x,y)\hat{\qT}d_B(y,z)\subseteq d_B(x,z)$, since

    $d_B(x,y)\hat{\qT}d_B(y,z)=$ by definition of $d_B$\\
    $(\qJ\setb{\qLowerSet R_0}{(x,y)\in R_0\in B})\hat{\qT}(\qJ\setb{\qLowerSet R_1}{(y,z)\in R_1\in B})=$ by distributivity\\
    $(\qJ\setb{(\qLowerSet R_0)\hat{\qT}(\qLowerSet R_1)}{(x,y)\in
      R_0\in B\land (y,z)\in R_1\in B}) \subseteq$ by definition of $\hat{\qT}$\\
    $(\qJ\setb{(\qLowerSet(R_0\qT R_1)}{(x,y)\in R_0\in B\land (y,z)\in R_1\in B})\subseteq$ by definition of $\qT$\\
    $(\qJ\setb{\qLowerSet R}{(x,z)\in R\in B}$.
  \end{itemize}
  
  Finally, we prove that $(X,\QU)=\Phi_u(X,d_B,Q_B)$.  Consider the
  prime-base $P\defeq\setb{\qLowerSet R}{R\in B}$ for $Q_B$ (see
  Theorem~\ref{thm:apb:pcl}). Thus, $P_\qj$ is a base for $Q_B$.
  We prove that $\forall R\in\QU.\exists\delta\in P.R_\delta\subseteq
  R$ and $\forall\delta\in P_\qj.R_\delta\in\QU$.  Since,
  $R_\bot=X^2$ and $R_{\delta_1\qj\delta_1}=R_{\delta_1}\cap
  R_{\delta_2}$ in the second proof obligation $P_\qj$ can be
  replaced by $P$.
  \begin{itemize}
  \item For every $R\in\QU$, there exists $R_0\in B$ such that
    $R\supseteq R_0$ (by definition of base for $\QU$) and also
    $S\in B$ such that $R_0\prec S$ (by
    Theorem~\ref{thm:prec_properties} item~\ref{thm:prec:directed}).
    We prove that $R_\delta\subseteq R$ when
    $\delta\defeq\qLowerSet S\in P$, by the following chains of
    implications:

    $(x,y)\in R_\delta\iff\delta\ll d_B(x,y)\implies
    \delta\subseteq d_B(x,y)\implies$ since $R_0\in\delta$\\
    $R_0\in d_B(x,y)\implies (x,y)\in R_0\implies (x,y)\in R$.

  \item For every $\delta\in P$, there exists $R\in B\subseteq\QU$
    such that $\delta=\qLowerSet R$ (by definition of $P$) and also
    $S\in B$ such that $R\prec S$ (by
    Theorem~\ref{thm:prec_properties} item~\ref{thm:prec:directed}).
    We prove that $S\subseteq R_\delta$ (thus $R_\delta\in\QU$),
    by the following chains of implications

    $(x,y)\in S\implies\qLowerSet S\subseteq d_B(x,y)\implies$ by $R\prec S$\\
    $R\in d_B(x,y)\implies$ by definition of $\ll$\\
    $\delta\ll d_B(x,y)\iff (x,y)\in R_\delta$. \qedhere
  \end{itemize}
\end{proof}

\subsection{Metrization for Topologies}
\label{sec:met:top}

We prove that the full\&faithful functor $\Phi_c:\MS_c\rTo\Top$ is
surjective on objects, and therefore it is an equivalence.
One can follow two approaches:
\begin{itemize}
\item An indirect approach, exploiting the fact that the functor
  $T:\QUnif\rTo\Top$ is surjective on objects (see
  Remark~\ref{rmk:admit}).  Therefore, given $(X,\tau)\in\Top$, one
  can choose $(X,\QU)\in\QUnif$ such that $T(X,\QU)=(X,\tau)$, and any
  metrization $(X,d,Q)$ for $(X,\QU)$, \eg, the one obtained by the
  construction in Section~\ref{sec:met:qu}, will be also a metrization
  for $(X,\tau)$.

\item A direct approach, that given $(X,\tau)\in\Top$ defines
  $(X,d,Q)\in\MS_c$ such that $\Phi_c(X,d,Q)=(X,\tau)$. More
  precisely, we start from a sub-base $S$ for $\tau$ and take as $Q$
  the prime-algebraic locale $(\PS(S),\subseteq)$.  Note that
  $(\PS(S),\subseteq)$ is an instance of the prime-continuous lattice
  $\RD(B,\prec)$ in Definition~\ref{def:apb} by taking as abstract
  prime-base $(B,\prec)$ the flat poset $(S,=)$.

\end{itemize}

\begin{theorem}\label{thm:met:top}
  If $S$ is a sub-base for a topology $\tau$ on $X$, let $Q_S$ be the
  continuous locale $\RD(S,=) = (\PS(S),\subseteq)$ and $d_S:X^2\to Q_S$ be
  \begin{equation*}
    d_S(x,y)\defeq\setb{O\in S}{x\in O\implies y\in O}.
  \end{equation*}
  Then, $(X,d_S,Q_S)\in\MS_c$ and $(X,\tau)=\Phi_c(X,d_S,Q_S)$.
\end{theorem}

\begin{proof}
  First, the locale $Q_S\defeq\RD(S,=)=(\PS(S),\subseteq)$ is
  prime-algebraic, \ie, $\setb{q\in Q_S}{q\lll q}=\setb{\setf{O}}{O\in
    S}$ is the smallest prime-base for $Q_S$, and the set of compact
  elements $\setb{q\in Q_S}{q\ll q}=\PS_f(S)=\qwb S$ is the smallest
  base for $Q_S$.
  Second, $d_S$ is a metric, namely:
  \begin{itemize}
  \item $d_S(x,x)=S$, since $\forall O\in S.x\in O\implies x\in O$.

  \item $d_S(x,y)\cap d_S(y,z)\subseteq d_S(x,z)$, since:

    $d_S(x,y)\cap d_S(y,z)=$ by definition of $d_S$\\
    $\setb{O\in S}{(x\in O\implies y\in O)\land(y\in O\implies z\in O)}
    \subseteq$\\
    $\setb{O\in S}{x\in O\implies z\in O}=d_S(x,z)$.
  \end{itemize}
  Finally, we prove $\tau=\tau_{d_S}$, \ie, $\forall x\in X.\forall
  F\in\qwb S.B(x,F)\in\tau$ and $\forall O\in\tau.\forall x\in
  O.\exists F\in\qwb S.B(x,F)\subseteq O$.
  \begin{itemize}
  \item If $x\in X$ and $F\in\qwb S=\PS_f(S)$, then we have
  $$B(x,F)=\setb{y\in X}{F\subseteq d_S(x,y)}=
  \setb{y\in X}{\forall O\in F.x\in O\implies y\in O}=
  \bigcap_{O\in F, x\in O} O$$.
  Thus, $B(x,F)\in\tau$, because it is the intersection of finitely many
  elements in the sub-base $S$ of $\tau$.  

\item If $x\in O\in\tau$, then (by $S$ being a sub-base) there exists
  $F\subseteq_f S$ such that $x\in\cap F\subseteq O$, thus
  $B(x,F)=\cap F\subseteq O$. \qedhere
  \end{itemize}
\end{proof}

The indirect approach can produce qualitatively different quantales, as
shown by the following example.
\begin{example}
  Let $\ell^\infty$ be the set of bounded sequences of real numbers,
  \ie, sequences $(x_i|i\in\omega)$ such that
  $\sup_{i\in\omega}|x_i|<\infty$, then $d_\infty(x,y)\defeq
  \sup_{i\in\omega} |x_i-y_i|$ is a $\RQ_+$-metric, \ie,
  $(\ell^\infty,d_\infty,\RQ_+)\in\MS_u$.
  If $\tau_\infty$ is the topology on $\ell^\infty$ induced by
  $d_\infty$ and $\QU_\infty$ is the quasi-uniformity on
  $\ell^\infty$ induced by $d_\infty$, then
  \begin{itemize}
  \item $\tau_\infty$ does not have a countable sub-base. Hence, for
    any choice of sub-base $S$, the construction in
    Theorem~\ref{thm:met:top} will produce a prime-algebraic locale
    $Q_S$ with a cardinality larger than that of $\RQ_+$ and no
    countable (prime-)base.
    
\item $\QU_\infty$ on the other hand has a countable base, \eg,
  \begin{equation*}
    B=\setb{R_q}{q\in\Q_{>0}}\mbox{ where }
    R_q\defeq\setb{(x,y)\in\ell_\infty \times \ell_\infty}{d_\infty(x,y)<q}
  \end{equation*}
    closed under binary composition of relations, namely $R_p\qT
    R_q=R_{p+q}$.  By applying the construction in
    Theorem~\ref{thm:met:qu} with this particular base we get a
    quantale isomorphic to $\RQ_+$.
  \end{itemize}
\end{example}

Metrization results for topological spaces---using suitable notions of
a generalized metric---were first proved
in~\parencite{Kopperman:All_topologies_metric:1988} and then refined
by others, most notably,
Flagg~\parencite{Flagg:Quantales_continuity_spaces:1997}. In~\parencite[Section
4.2]{Dagnino_Farjudian_Moggi:Robust_Topology:arXiv:2025}, we show that
Flagg's result, where open balls are defined using the totally-below
relation, continues to hold when open balls are defined using the
way-below relation. However, the continuous quantales (or related
notions) used in these proofs are very big. To be more precise, the
cardinality of a base for such continuous quantales is usually as big
as the cardinality of the topology $\tau$. The construction of
Theorem~\ref{thm:met:top} results in a more manageable quantale which
has a countable (prime-)base when $S$ is countable, which is possible
when $\tau$ is second-countable.

\section{Powerset-Like Monads}
\label{sec:Powerset-like-Monads}

In the previous section, we proved that the functor
$\Phi_u:\MS_u\rTo\QUnif$ is an equivalence. Therefore, any monad on
one of the categories $\MS_u$ or $\QUnif$ induces a monad on the
other.  The following diagram recalls some of the relations depicted
in Diagrams~\eqref{diag:forget2}~and~\eqref{diag:Met_QU_Top}:
\begin{equation}
  \label{diag:Metu_QU_Po_Set}
  \begin{tikzcd}[row sep = large, column sep = huge]
    \MS_u \arrow[r, rightarrow, "\Phi_u"] &
        \QUnif \arrow[r, "T"'] \arrow[rd, "U"'] &
    \Top \arrow[l,bend right,dashed,"\bot","L_T"'] \arrow[d, "U"] \arrow[r, "P"']&
    \Po \arrow[l,bend right,dashed,"\bot","L_P"'] \arrow[ld, "U"] \\
    & & \Set
  \end{tikzcd}
\end{equation}
In this section, we define three monads on $\QUnif$ that are
\emph{liftings} of the powerset monad on $\Set$ along the functor
$\QUnif\rTo\Set$.
We adopt a presentation of monads in terms of Kleisli
triple (see~\parencite{Manes1976,Moggi:notions_monads:1991}).

\begin{definition}\label{def:monad}
  A monad on a category $\A$ is a triple $\hat{M}=(M,\eta,-^*)$,
  where:
 \begin{itemize}
 \item $M$ is a function on the objects of $\A$,
 \item $\eta$ is a family of arrows $\eta_X:\A(X,MX)$ with $X\in\A$,
 \item $-^*$ is a family of maps $\A(X,MY)\to\A(MX,MY)$ with
   $X,Y\in\A$,
 \end{itemize}
 satisfying the following equations:
 \begin{equation}\label{eq:monad}
   f^*\circ\eta_X=f\quad,\quad
   \eta_X^*=\id_{MX}\quad,\quad
   g^*\circ f^*=(g^*\circ f)^*.
 \end{equation}
\end{definition}

We recall the bijective correspondence between monads $(M,\eta,\mu)$ on $\A$,
where $M$ is an endofunctor on $\A$ and $\eta$ and $\mu$ are natural
transformations, and Kleisli triples $(M,\eta,-^*)$:
\begin{itemize}
\item if $f : \A(X,Y)$, then $Mf\defeq(\eta_Y\circ f)^*:\A(MX,MY)$
\item $\mu_X\defeq\id_{MX}^*:\A(M^2X,MX)$, and
\item if $f:\A(X,MY)$, then $f^*=\mu_Y\circ Mf:\A(MX,MY)$.
\end{itemize}

For our purposes, we use a strict notion of monad map, namely a
functor between the underlying categories, which \emph{commutes} with
the monad structures.

\begin{definition}\label{def:monad:lift}
  A monad $\hat{M}$ on $\A$ is a \textbf{lifting} of the monad
  $\hat{M'}$ on $\A'$ along the functor $U:\A\rTo\A'$ $\defiff$
  \begin{equation}
    \label{eq:monad:lift}
    U(MX)=M'(UX)\quad,\quad
    U(\eta_X)=\eta'_{UX}\quad,\quad
    U(f^*)=(Uf)^{*'}.
  \end{equation}
\end{definition}
  
\begin{remark}
  If $U:\A\rTo\A'$ is faithful, then there can be several lifting of
  $\hat{M'}$ along $U$.  However, if $\hat{M}_1$ and $\hat{M}_2$ are
  two liftings of $\hat{M}'$ such that $M_1=M_2$, then the other
  components of the triples $\hat{M}_1$ and $\hat{M}_2$ are also
  equal.
  If $\A$ is a sub-category of $\A'$, then there is at most one
  lifting $\hat{M}$ of $\hat{M'}$ along the inclusion $\A\rInto\A'$.
  In this case, it is more appropriate to call $\hat{M}$ the
  \emph{restriction} of $\hat{M'}$ to $\A$.
\end{remark}

We recall the definition of the powerset monad $\hat{\PM}$ on $\Set$.
\begin{definition}
  \label{def:PM}
  The powerset monad $\hat{\PM}$ on $\Set$ is given by the triple
  $(\PM,\eta,-^*)$, where:
  \begin{itemize}
  \item $\PM(X)$ is the set of subsets of $X$,
  \item $\eta_X:\Set(X,\PM(X))$ is $\eta_X(x)=\{x\}$,
  \item if $f:\Set(X,\PM(X'))$, then $f^*:\Set(\PM(X),\PM(X'))$ is
    $f^*(A)=\bigcup\setf{f(x) \mid x\in A}$.
  \end{itemize}
\end{definition}

We now define three liftings of the powerset monad on $\Set$ along the
forgetful functor $U:\QUnif\rTo\Set$.
\begin{definition}\label{def:PM:QU}
Given a base $B$ for a quasi-uniformity $\QU$ on $X$, we define three
quasi-uniformities on $\PM(X)$:
\begin{itemize}
\item the Hausdorff-Smyth quasi-uniformity $\QU_S$ is generated by the
  base $B_S\defeq\setb{R_S}{R\in B}$, where
  \begin{equation*}
    R_S\defeq\setb{(A_1,A_2)\in\PM(X)^2}{ A_2 \subseteq R(A_1)}=
    \setb{(A_1,A_2)\in\PM(X)^2}{\forall x_2\in A_2.\exists
    x_1\in A_1.(x_1,x_2)\in R}.
  \end{equation*}
\item the Hausdorff-Hoare quasi-uniformity $\QU_H$ is generated by the
  base $B_H\defeq\setb{R_H}{R\in B}$, where
  \begin{equation*}
    R_H\defeq((R^o)_S)^o=\setb{(A_1,A_2)\in\PM(X)^2}{A_1 \subseteq R^o(A_2)}.
  \end{equation*}
\item the Hausdorff-Plotkin quasi-uniformity $\QU_P$ is generated by the base
  $B_P\defeq\setb{R_P}{R\in B}$, where
  \begin{equation*}
    R_P\defeq R_H\cap R_S=
    \setb{(A_1,A_2)\in\PM(X)^2}{A_2\subseteq R(A_1)\land A_1 \subseteq R^o(A_2)}.
  \end{equation*}
\end{itemize}
For each $\alpha\in\setf{H,S,P}$, we write $\PM_\alpha$ for the map on
quasi-uniform spaces such that
$\PM_\alpha(X,\QU)=(\PM(X),\QU_\alpha)$.
\end{definition}

\begin{proposition}\label{prop:PM:QU}
  For each $\alpha\in\setf{H,S,P}$, the quasi-uniform space
  $\PM_\alpha(X,\QU)$ is well-defined.
\end{proposition}
\begin{proof}
  First, observe that, by Proposition~\ref{prop:lax}, every
  lax-monoidal map $F:\PS(X^2)\to\PS(\PS(X)^2)$ maps every base $B$
  (\ie, satisfying the properties required in
  Proposition~\ref{prop:qu:base}) for a quasi-uniformity $\QU$ on $X$
  to a base $F(B)\defeq\setb{F(R)}{R\in B}$ for the same
  quasi-uniformity on $\PS(X)$, {\ie}, if $B$ and $B'$ are bases for
  the same quasi-uniformity on $X$, then $F(B)$ and $F(B')$ are bases
  for the same quasi-uniformity on $\PS(X)$.
  Since, by Proposition~\ref{prop:BR:liftS}~and~\ref{prop:BR:main},
  each $-_\alpha:\PS(X^2)\to\PS(\PS(X)^2)$ is lax-monoidal, we
  conclude that each $B_\alpha$ is a base for a quasi-uniformity
  $\QU_\alpha$ on $\PS(X)$.
\end{proof}

\begin{remark}
  The quasi-uniformities of Definition~\ref{def:PM:QU} appears
  in the literature with a different terminologies. For instance,
  in~\parencite{RodriguezLopez_Romaguera:Vietoris_Hausdorff:2002}:
\begin{itemize}
\item The Hausdorff-Smyth quasi-uniformity $\QU_S$ is called the lower
  Hausdorff quasi-uniformity.

\item The Hausdorff-Hoare quasi-uniformity $\QU_H$ is called the upper
  Hausdorff quasi-uniformity.

\item The Hausdorff-Plotkin quasi-uniformity $\QU_P$ is called the
  Hausdorff (or Bourbaki) quasi-uniformity.
\end{itemize}
\end{remark}

\begin{lemma}
  \label{lem:SH:qu}
  $(\QU_H)^o=(\QU^o)_S$ and $\QU_P=\QU_S\qj\QU_H$, \ie, the join in
  $\QUnif_{\PS(X)}$ of $\QU_S$ and $\QU_H$.
\end{lemma}

\begin{proof}
  By definition of $\QU_\alpha$, we have, $\forall
  R_0\subseteq\PS(X)^2.R_0\in\QU_\alpha\iff\exists
  R\in\QU.R_\alpha\subseteq R_0$.
  For the first claim we have:
  \begin{itemize}
  \item $R_0^o\in(\QU_H)^o\iff R_0\in\QU_H$ by definition of $-^o$ on
    quasi-uniformities,
  \item $R_0\in\QU_H\iff\exists R\in\QU.(R_H)^o\subseteq R_0^o$ by
    definition of $\QU_H$ and $-^o$ involution on $\PS(\PS(X)^2)$,
  \item $(R_H)^o=(R^o)_S$ by definition of $R_H$.
  \end{itemize}
  Therefore,
  $R_0\in(\QU_H)^o\iff R_0^o\in\QU_H\iff
  \exists R\in\QU.(R^o)_S\subseteq R_0\iff \exists
  R\in\QU^o.R_S\subseteq R_0\iff R_0\in(\QU^o)_S$.

  The second claim follows from Remark~\ref{rmk:quX:CL}, which gives a
  description of a base for the join of finitely many
  quasi-uniformities in $\QUnif_X$.
\end{proof}

\begin{theorem}
  \label{thm:PM:qu}
  For each $\alpha\in\setf{H,S,P}$, the triple $(\PM_\alpha,\eta,-^*)$
  is a lifting of $(\PM,\eta,-^*)$ along $U:\QUnif\rTo\Set$.
\end{theorem}

\begin{proof}
  We have to verify that:
  \begin{enumerate}
  \item $U(\PM_\alpha(X,\QU))=\PM(X)$, which is obvious.
  \item $\eta:X\rTo\PM(X)$ is in $\QUnif((X,\QU),\PM_\alpha(X,\QU))$,
    \ie,
    $\forall R\in\QU.\setb{(x_1,x_2)}{(\setf{x_1},\setf{x_2})\in
      R_\alpha}\in\QU$, which follows from the stronger property
    $\forall R\in\QU.\forall x_1,x_2\in X.(x_1,x_2)\in R\iff
    (\setf{x_1},\setf{x_2})\in R_\alpha$.
  \item If
    $f : \QUnif((X,\QU),\PM_\alpha(X',\QU'))\subseteq\Set(X,\PM(X'))$,
    then $f^* : \QUnif(\PM_\alpha(X,\QU),\PM_\alpha(X',\QU'))$.
  \end{enumerate}

  First, we prove the third property for $\PM_H$, that is, if $f$ is
  quasi-uniformly continuous (\ie,
  $\forall R'\in\QU'.f^\dagger(R'_H)\in\QU$), then so is $f^*$. More
  precisely, we prove that $R_H\subseteq (f^*)^\dagger(R'_H)$ for
  every $R'\in\QU'$, where $R\defeq f^\dagger(R'_H)\in\QU$.

  \begin{itemize}
  \item $(A_1,A_2)\in R_H\iff$ by definition of $\R_H$
  \item $\forall x_1\in A_1.\exists x_2\in A_2.(x_1,x_2)\in R\iff$
    by definition of $R=f^\dagger(R'_H)$
  \item
    $\forall x_1\in A_1.\exists x_2\in A_2.(f(x_1),f(x_2))\in
    R'_H\iff$ by definition of $\R'_H$

\item $\forall x_1\in A_1.\exists x_2\in A_2.\forall y_1\in
    f(x_1).\exists y_2\in f(x_2).(y_1,y_2)\in R'\iff$ by definition of
    $f^*(A)=\cup_{x\in A}f(x)$
  \item $\forall y_1\in f^*(A_1).\exists y_2\in
    f^*(A_2).(y_1,y_2)\in R'\iff$ by definition of $\R'_H$
  \item $(f^*(A_1),f^*(A_2))\in R'_H$.
  \end{itemize}
  Next, we use Lemma~\ref{lem:SH:qu} to prove the third property for
  $\PM_S$ from that for $\PM_H$.
  \begin{itemize}
  \item $f : \QUnif((X,\QU),\PM_S(X',\QU'))\iff$ by definition of
    $\PM_S(X',\QU')$
  \item $f : \QUnif((X,\QU),(\PM(X'),\QU'_S))\iff$ by applying $-^o$
    (and $f^o=f$ as a set-theoretic map)
  \item $f : \QUnif((X,\QU^o),(\PM(X'),(\QU'_S)^o))\iff$
    because $(\QU_H)^o=(\QU^o)_S$
  \item $f : \QUnif((X,\QU^o),(\PM(X'),(\QU')^o)_H))\implies$ by the third
    property for $\PM_H$
  \item $f^* : \QUnif((\PM(X),(\QU^o)_H),(\PM(X'),(\QU')^o)_H))\iff$
    because $(\QU_S)^o=(\QU^o)_H$  (by Lemma~\ref{lem:SH:qu})
  \item $f^* : \QUnif((\PM(X),(\QU_S)^o),(\PM(X'),(\QU'_S)^o))\iff$ by
    applying $-^o$
  \item $f^* : \QUnif((\PM(X),\QU_S),(\PM(X'),\QU'_S))$.
  \end{itemize}
  Finally, for $\PM_P$ we have:
  \begin{itemize}
  \item $f : \QUnif((X,\QU),\PM_P(X',\QU'))\implies$ since
    $\QU_H,\QU_S\subseteq\QU_P$ (by Lemma~\ref{lem:SH:qu})
  \item $f : \QUnif((X,\QU),\PM_H(X',\QU'))\land
    f : \QUnif((X,\QU),\PM_S(X',\QU'))\implies$ as $\PM_H$ and
    $\PM_S$ satisfy the third property
  \item $f^* : \QUnif(\PM_H(X,\QU),\PM_H(X',\QU'))\land
    f^* : \QUnif(\PM_S(X,\QU),\PM_S(X',\QU'))$.
  \end{itemize}
  To conclude that $f^* : \QUnif(\PM_P(X,\QU),\PM_P(X',\QU'))$, we
  show that $\forall R'\in\QU'.(f^*)^\dagger(R'_P)\in\QU_P$:
  \begin{itemize}
  \item $(f^*)^\dagger(R'_P)=$ by definition of $R'_P$ and because
    inverse image preserves intersections
  \item $(f^*)^\dagger(R'_H)\cap(f^*)^\dagger(R'_S)\in\QU_P$ because
    $\QU_P$ is a filter and
    $(f^*)^\dagger(R'_\alpha)\in\QU_\alpha\subseteq\QU_P$ for
    $\alpha=H,S$. \qedhere
  \end{itemize}
  \end{proof}
  
  As already mentioned, the equivalence $\Phi_u:\MS_u\rTo\QUnif$ and
  the monads $\hat{\PM}_\alpha$, for $\alpha\in\{S,H,P\}$, on the
  category $\QUnif$ induce \emph{equivalent} monads on the category
  $\MS_u$.
  We now prove that the Hausdorff-Smyth monad $\hat{\PM}_S$ on $\MS_u$
  defined in
 ~\parencite[Definition~6.7]{Dagnino_Farjudian_Moggi:Robust_Topology:arXiv:2025}
  is a lifting of the monad $\hat{\PM}_S$ on $\QUnif$ along
  $\Phi_u:\MS_u\to\QUnif$.  Recall that the action of $\PM_S$ on
  an object $(X,d,Q):\MS_u$ is $\PM_S(X,d,Q)=(\PS(X),d_S,Q_S)$, where
  $Q_S$ is the quantale on the continuous lattice $\CM(Q)$ of Scott
  closed subsets of $Q$ (see
 ~\parencite[Definition~6.5]{Dagnino_Farjudian_Moggi:Robust_Topology:arXiv:2025})
  such that
\begin{itemize}
\item the  unit is $\hat{\qI}=\cl{\setf{\qI}}=\qLowerSet \qI$ (where $\qI$ is
    the unit of the quantale $Q$),
  \item the tensor product of two Scott closed subsets $A_1$ and $A_2$
    of $Q$ is given by the Scott closure of the pointwise
    multiplication, {\ie},
    $A_1\qT_\CM A_2= \cl{\setb{q_1\qT q_2}{q_1\in A_1\land q_2\in
        A_2}}$,
  \end{itemize}
  and the metric $d_S$ is defined by
  $d_S(A_1,A_2)=\bigcap_{x_2\in A_2}\cl{\setb{d(x_1,x_2)}{x_1 \in
      A_1}}$ ((see
 ~\parencite[Definition~6.7]{Dagnino_Farjudian_Moggi:Robust_Topology:arXiv:2025}).

  \begin{theorem}
    \label{thm:HS_Lifting_PS_monad}
  The Hausdorff-Smyth monad $\hat{\PM}_S$ on $\MS_u$ is a lifting of
  the powerset monad $\hat{\PM}_S$ on $\QUnif$ along the equivalence
  $\Phi_u:\MS_u\rTo\QUnif$.
\end{theorem}

\begin{proof}
  Since the equivalence $\Phi_u$ is faithful, it suffices to prove
  that $\Phi_u(\PM_S(X,d,Q))=\PM_S(\Phi_u(X,d,Q))$, \ie,
  $\QU_{d_S}=(\QU_d)_S$, for every $(X,d,Q)\in\MS_u$.
  Recall that:
  \begin{enumerate}
     \item If $B$ is a base for $Q$, then a possible base for $\CM(Q)$ is
     $B_C\defeq\setb{\qLowerSet F}{F\subseteq_f B}$ and $\qLowerSet
     F\in\qwb_{B_C} C\iff F\subseteq_f\bigcup\setb{\qwb_B q}{q\in C}$.

   \item The following properties hold in the continuous lattice $\CM(Q)$:
     \begin{itemize}
     \item $\forall C\in\CM(Q).C=\cl{\qwb C}$, where $\qwb
       A\defeq\bigcup\setb{\qwb q}{q\in A}\subseteq\qLowerSet
       A\subseteq\cl{A}$ for $A\subseteq Q$;
     \item $\qwb(\bigcap\setb{C_i}{i\in I})=\qwb(\bigcap\setb{\qwb
       C_i}{i\in I})$, where $C_i\in\CM(Q)$ for every $i\in I$;
     \item $\forall F\subseteq_f Q.\forall C\in\CM(Q).\qLowerSet F\ll
       C\iff F\subseteq\qwb C$.
     \end{itemize}
     Thus, $\qLowerSet F\in\qwb_{B_C}d_S(A_1,A_2)\iff
     F\subseteq_f\qwb_B(\bigcap\setb{\qwb_B\setb{d(x_1,x_2)}{x_1 \in
         A_1}}{x_2\in A_2})$.

   \item The quasi-uniformity $\QU_d$ is generated by the base
     $B^u_d=\setb{R_\delta}{\delta\in\qwb_B\qI}$, where $B$ is a base
     for $Q$ and $R_\delta=\setb{(x_1,x_2)\in
       X^2}{\delta\ll d(x_1,x_2)}$.
     Thus, using the base $B_C$ for $\CM(Q)$ given in the previous
     item, we have that $\QU_{d_S}$ is generated by
     $(B_C)^u_{d_S}=\setb{R_{\qLowerSet F}}{F\subseteq_f\qwb_B\qI}$, where
     $R_{\qLowerSet F}=\setb{(A_1,A_2)\in\PS(X)^2}{\qLowerSet F\ll
       d_S(A_1,A_2)}$.

  \item If $B$ is a base for a quasi-uniformity $\QU$ on $X$, then the
    quasi-uniformity $\QU_S$ on $\PS(X)$ is generated by the base
    $B_S\defeq\setb{R_S}{R\in B}$, where
    $R_S\defeq\setb{(A_1,A_2)\in\PM(X)^2}{\forall x_2\in A_2.\exists
      x_1\in A_1.(x_1,x_2)\in R}$.
    Thus, $(\QU_d)_S$ is generated by the base
    $(B^u_d)_S=\setb{(R_\delta)_S}{\delta\in\qwb_B\qI}$, where $B$ is
    base for $Q$.
  \end{enumerate}
  To prove that $\QU_{d_S}=(\QU_d)_S$, it suffices to show that their
  bases $(B_C)^u_{d_S}$ and $(B^u_d)_S$ satisfy the following:
  \begin{itemize}
  \item For every $F\subseteq_f\qwb_B\qI$, the entourage
    $R_{\qLowerSet F}=\setb{(A_1,A_2)\in\PS(X)^2}{\qLowerSet F\ll
      d_S(A_1,A_2)}\in (B_C)^u_{d_S}$ contains the entourage
    $(R_\delta)_S=\setb{(A_1,A_2)\in\PM(X)^2}{\forall x_2\in
      A_2.\exists x_1\in A_1.\delta\ll d(x_1,x_2)}\in (B^u_d)_S$ for
    some $\delta\in\qwb_B\qI$.
    
    If $F\subseteq_f\qwb_B\qI$, then $\qJ F\in\qwb_B\qI$, because $B$
    is closed under finite joins and $\qwb_B\qI$ is directed.  By the
    interpolation property for $\ll$, there exists $\delta\in B$ such
    that $\qJ F\ll\delta\ll\qI$.  We claim that such a $\delta$ has the
    required property.
    In fact, $\forall x_2\in A_2.\exists x_1\in A_1.\delta\ll
    d(x_1,x_2)\implies\delta\in d_S(A_1,A_2)$, thus $\qLowerSet
    F\subseteq\qLowerSet\setf{\qJ
      F}\ll\qLowerSet\setf{\delta}\subseteq d_S(A_1,A_2)$.

     \item For every $\delta\in\qwb_B\qI$, the entourage
       $(R_\delta)_S=\setb{(A_1,A_2)\in\PM(X)^2}{\forall x_2\in
         A_2.\exists x_1\in A_1.\delta\ll d(x_1,x_2)}\in (B^u_d)_S$
       contains the entourage
       $R_{\qLowerSet F}=\setb{(A_1,A_2)\in\PS(X)^2}{\qLowerSet F\ll
         d_S(A_1,A_2)}\in B^u_{d_S}$ for some $F\subseteq_f\qwb_B\qI$.

    We claim that $F=\setf{\delta}$ has the required property.  In
    fact, we have the following chain on implications
    $\qLowerSet\setf{\delta}\in\qwb_{B_C} d_S(A_1,A_2)\iff
    \delta\in\qwb_B d_S(A_1,A_2)\implies\forall x_2\in A_2.\exists
    x_1\in A_1.\delta\ll d(x_1,x_2)$. \qedhere
\end{itemize}
\end{proof}

Since the functor $U:\QUnif\to\Set$ factors through $\Top$ and $\Po$,
one may wonder whether the monads $\PM_\alpha$, for
$\alpha\in\{S,H,P\}$, on $\QUnif$ arise as liftings along
$T:\QUnif\to\Top$ or $P\circ T:\QUnif\to\Po$ of other liftings of the
powerset monad on $\Set$ along $U:\Top\to\Set$ or $U:\Po\to\Set$.  We
already know from~\parencite[Example
4.18]{Dagnino_Farjudian_Moggi:Robust_Topology:arXiv:2025}, however,
that this is not the case for $T : \QUnif\to\Top$ because there are
metric spaces, and thus quasi-uniform spaces, inducing the same
topology which are mapped by $\PM_\alpha$ to metric spaces inducing
different topologies.
The next example shows that something similar happens for $P\circ
T:\QUnif\to\Po$.

\begin{example}\label{ex:not-factor}
  We define two quasi-uniformities $\QU$ and $\QV$ on $\omega$ having
  the same specialization preorder, \ie, $\leq_\QU=\leq_\QV$, while
  the quasi-uniformities $\QU_S$ and $\QV_S$ on $\PS(\omega)$ have
  different specialization preorders.
  Let $B = \setb{R_n\subseteq \omega^2}{n>0}$, where
  $R_n \defeq \Delta_\omega \cup
  \setb{(x_1,x_2)\in\omega^2}{x_2=0\land x_1\geq n}$ for every
  $n\in\omega$ Then, $B$ is a base for a quasi-uniformity $\QU$ on
  $\omega$.  Since $\bigcap B = \Delta_\omega$, we have
  $\leq_\QU=\Delta_\omega$, \ie, the equality relation on $\omega$.
  Let $\QV$ be the principal filter in $\PS(\omega^2)$ generated by
  $\Delta_\omega$, and therefore $\leq_\QV=\Delta_\omega$.  It is easy
  to see that $P(T(\PM_S(\omega,\QV)))=(\PS(\omega),\leq_{\QV_S}) =
  (\PS(\omega),\supseteq)$.
  On the other hand,
  $P(T(\PM_S(\omega,\QU)))=(\PS(\omega),\leq_{\QU_S})$, where $\QU_S$
  is generated by $B_S = \setb{(R_n)_S\subseteq \PS(\omega)^2}{n>0}$
  with $(A_1,A_2)\in (R_n)_S \defiff (A_2\setminus\setf{0})\subseteq A_1
  \land (0\in A_2\implies \exists x\in A_1.(x=0\lor x\geq n))$.
  In particular, $(\setb{n}{n>0},\setf{0})\in \bigcap B_S$, \ie,
  $\setb{n}{n>0}\leq_{\QU_S}\setf{0}$, while
  $\setb{n}{n>0}\leq_{\QV_S}\setf{0}$ is false.
\end{example}

By general considerations on the relation between monads and
adjunctions, an adjunction
$\begin{tikzcd}[row sep =large, column sep = huge]
  \A \arrow[r, "R"'] & \B \arrow[l,bend right,dashed,"\bot","L"']
\end{tikzcd}$
induces a \emph{monad transformer} turning every monad on $\A$ into a
monad on $\B$ (see~\parencite{JMoggi2010monad}).  Since $\Top$ and $\Po$
are coreflective sub-categories of $\QUnif$, the monads
$\hat{\PM}_\alpha$, for $\alpha\in\{S,H,P\}$, induce monads
$\hat{\PM}_\alpha$ on $\Top$ and $\Po$ whose action on objects are
\begin{itemize}
\item $\PM_\alpha(X,\tau) \defeq T(\PM_\alpha(L_T(X,\tau)))
  =T(\PS(X),(\QU_\tau)_\alpha)$ and
\item $\PM_\alpha(X,\leq)\defeq
  (P\circ T)(\PM_\alpha((L_T\circ L_P)(X,\leq)))
  = P(T(\PS(X),(\QU_\leq)_\alpha))$.
\end{itemize}
We prove that the monads $\hat{\PM}_\alpha$ on $\Po$ are the
restrictions of the corresponding monad on $\QUnif$ along the
$(P\circ T)$-section $L_T\circ L_P:\Po\to\QUnif$, which allows viewing
$\Po$ as a full sub-category of $\QUnif$.
\begin{lemma}
  If $(X,\leq)\in\Po$ and $-_\alpha:\PS(X^2)\to\PS(\PS(X)^2)$ is any
  of the maps introduced in Definition~\ref{def:PM:QU}, then
  $(\PS(X),\leq_\alpha)\in\Po$.
\end{lemma}
\begin{proof}
  Since each $-_\alpha:\PS(X^2)\to\PS(\PS(X)^2)$ is a lax-monoidal map
  (see Propositions~\ref{prop:BR:liftS}~and~\ref{prop:BR:main}), and
  reflexivity and transitivity are among the properties of binary
  relations preserved by such maps (see Proposition~\ref{prop:lax}),
  if $\leq$ is a preorders then so is $\leq_\alpha$.
\end{proof}

\begin{remark}
  We call $\leq_S$, $\leq_H$ and $\leq_P$ the Smyth, Hoare, and
  Plotkin preorders, respectively.  We adopt these names due to the
  similarity of these preorders with the relations on abstract bases
  used to define Smyth, Hoare, and Plotkin powerdomains (see
 ~\parencite[Proposition~6.2.9 and Theorem~6.2.10]{AbramskyJung94-DT}).
\end{remark}

\begin{theorem}
  \label{thm:monads_Po_QUnif}
For every $(X,\leq)\in\Po$ and $\alpha\in\{S,H,P\}$, we have
$\PM_\alpha(X,\QU_\leq) =(\PS(X),\QU_{\leq_\alpha})$.
\end{theorem}
\begin{proof}
  By Proposition~\ref{prop:aqu}, $\QU_\leq=\upperSetOf{\leq}$. Hence,
  $B=\setf{\leq}$ is a base for the quasi-uniformity $\QU_\leq$.
  By Definition~\ref{def:PM:QU}, $B_\alpha=\setf{\leq_\alpha}$ is a
  base for the quasi-uniformity $(\QU_\leq)_\alpha$.  Therefore,
  $(\QU_\leq)_\alpha=\QU_{\leq_\alpha}$.
\end{proof}

\section{Concluding Remarks}
\label{sec:concluding_remarks}

We have established several relations, at a category-theoretic level,
between quasi-uniform spaces, topological spaces, and quantale-valued
metric spaces.  In particular, we have identified the category of
quasi-uniform spaces and quasi-uniformly continuous maps as the
appropriate qualitative counterpart of the category of quantale-valued
metrics and uniformly continuous maps, thus providing a framework for
qualitative robustness analysis. We have also shown that the
Hausdorff-Smyth monad on quantale-valued metric spaces arises as a
lifting of a corresponding powerset-like monad on quasi-uniform spaces.

The present treatment has deliberately ignored order-enriched
structure. The category $\QUnif$ is naturally $\Po$-enriched, and the
Hausdorff-Smyth, Hausdorff-Hoare, and Hausdorff-Plotkin monads are
$\Po$-enriched as well. Incorporating this enrichment would make it
possible to apply the framework developed
in~\parencite[Section~5]{Dagnino_Farjudian_Moggi:Robust_Topology:arXiv:2025}
and obtain separated variants of the corresponding
monads. Furthermore, if $\A$ is a $\Po$-enriched category, one may
define $\A^o(X,Y)=\A(X,Y)^o$, that is, by reversing the preorder on
each hom-set. In this setting, the involution functor $\-^o$ on $\Po$
becomes a $\Po$-enriched functor from $\Po$ to $\Po^o$.

A natural direction for future work is the development of
suitable effective structures for quasi-uniform spaces and
quantale-valued metric spaces. Such structures would be needed to
connect the qualitative framework developed here with computability
and effective robustness analysis.

We do not know whether the counterpart of
Theorem~\ref{thm:monads_Po_QUnif} holds for $\Top$ or not. It will be
interesting to determine whether the three monads induced on $\Top$ by
those on $\QUnif$ are related to any well-known monads on $\Top$.



\printbibliography

@string{fss="Fuzzy Sets and Systems"}

@article{Manes1976,
  title={Algebraic theories in a category},
  author={Manes, Ernest G and Manes, Ernest G},
  journal={Algebraic Theories},
  pages={161--279},
  year={1976},
  publisher={Springer}
}

@book{Hart_et_al:Encyclopedia_Topology:2003,
  title={Encyclopedia of general topology},
  author={Hart, Klaas Pieter and Nagata, Jun-iti and Vaughan, Jerry E.},
  year={2003},
  publisher={Elsevier}
}

@book{borceux1994handbook1,
  title={Handbook of categorical algebra: Basic category theory},
  author={Borceux, Francis},
  volume={1},
  year={1994},
  publisher={Cambridge University Press}
}

@book{borceux1994handbook2,
  title={Handbook of categorical algebra: Categories and structures},
  author={Borceux, Francis},
  volume={2},
  year={1994},
  publisher={Cambridge University Press}
}

@article{Moggi:notions_monads:1991,
  title={Notions of computation and monads},
  author={Moggi, Eugenio},
  journal={Information and computation},
  volume={93},
  number={1},
  pages={55--92},
  year={1991},
  publisher={Elsevier}
}

@article{JMoggi2010monad,
  title={Monad transformers as monoid transformers},
  author={Jaskelioff, Mauro and Moggi, Eugenio},
  journal={Theoretical computer science},
  volume={411},
  number={51-52},
  pages={4441--4466},
  year={2010},
  publisher={Elsevier}
}

@Article{Moggi_Farjudian_Duracz_Taha:Reachability_Hybrid:2018,
  author = 	 {Eugenio Moggi and Amin Farjudian and Adam Duracz and
                  Walid Taha},
  title = 	 {Safe \& Robust Reachability Analysis of Hybrid Systems},
  fjournal = 	 {Theoretical Computer Science},
  journal = 	 {Theor. Comput. Sci.},
  volume = "747",
  pages = "75--99",
  year = 	 {2018},
  doi =          {10.1016/j.tcs.2018.06.020}
}

@article{lawvere1973metric,
  title={Metric spaces, generalized logic, and closed categories},
  author={Lawvere, F.~William},
  journal={Rendiconti del seminario mat{\'e}matico e fisico di Milano},
  volume={43},
  pages={135--166},
  year={1973},
  publisher={Springer},
  note={Reprints in Theory and Applications of Categories, No. 1, 1-37 (2002)}
}

@Book{Goubault-Larrecq:Non_Hausdorff_topology:2013,
  author = {Goubault-Larrecq, Jean},
  title = {Non-Hausdorff topology and domain theory},
  OPTseries = {New mathematical monographs: 22},
  publisher = {Cambridge University Press},  
  OPTisbn = {1107034132},
  year = {2013},
  OPTlanguage = {eng},
  OPTaddress = {Cambridge},  
  OPTkeywords = {Topology -- Problems exercises etc; Topology},
  OPTlccn = {2013427187}
}

@incollection{AbramskyJung94-DT,
  author =	 {Samson Abramsky and Achim Jung},
  booktitle =	 {Handbook of Logic in Computer Science},
  title =	 {Domain Theory},
  publisher =	 {Clarendon Press},
  address =      {Oxford},
  pages =	 {1--168},
  year =	 {1994},
  editor =	 {S. Abramsky and D. M. Gabbay and T. S. E. Maibaum},
  volume =	 {3}
}

@article{smyth1977effectively,
  title={Effectively given domains},
  author={Smyth, Michael B.},
  fjournal={Theoretical Computer Science},
  journal={Theor. Comput. Sci.},
  volume={5},
  number={3},
  pages={257--274},
  year={1977},
  publisher={Elsevier}
}

@Book{Gierz-ContinuousLattices-2003,
  author = 	 {Gerhard Gierz and Karl Heinrich Hofmann and Klaus Keimel and Jimmie D. Lawson and Michael W. Mislove and Dana S. Scott},
  title = 	 {Continuous Lattices and Domains},
  publisher = 	 {Cambridge University Press},
  year = 	 {2003},
  volume = 	 {93},
  series = 	 {Encycloedia of Mathematics and its Applications}
}

@article{CookW21,
  author       = {Derek S. Cook and
                  Ittay Weiss},
  title        = {The topology of a quantale valued metric space},
  journal      = fss,
  volume       = {406},
  pages        = {42--57},
  year         = {2021},
  doi          = {10.1016/j.fss.2020.06.005},
}

@Article{Farjudian_Moggi:Robustness_Scott_Continuity_Computability:2023,
  author = 	 {Amin Farjudian and Eugenio Moggi},
  title = 	 {Robustness, {Scott} Continuity, and Computability},
  journal={Mathematical Structures in Computer Science},
  DOI={10.1017/S0960129523000233},
  year = 	 {2023},
  pages={1–37}
}

@Article{Flagg:Quantales_continuity_spaces:1997,
  author = 	 {Flagg, R.C.},
  title = 	 {Quantales and continuity spaces},
  journal = 	 {Algebra Universalis},
  year = 	 {1997},
  OPTkey = 	 {},
  volume = 	 {37},
  number = 	 {3},
  pages = 	 {257--276},
  OPTmonth = 	 {},
  OPTnote = 	 {},
  OPTannote = 	 {},
  doi =          {10.1007/s000120050018}
}

@Book{Fletcher_Lindgren:Quasi_Uniform:Book:1982,
  author={Peter Fletcher and William F. Lindgren},
  title = 	 {Quasi-uniform spaces},
  publisher = 	 {Marcel Dekker},
  year = 	 {1982}
}

@article{Kopperman:All_topologies_metric:1988,
  title={All topologies come from generalized metrics},
  author={Kopperman, Ralph},
  journal={The American Mathematical Monthly},
  volume={95},
  number={2},
  pages={89--97},
  year={1988},
  publisher={Taylor \& Francis}
}

@phdthesis{gavazzo2019phd,
  title={Coinductive equivalences and metrics for higher-order languages with algebraic effects},
  author={Gavazzo, Francesco},
  school={University of Bologna, Italy},
  year={2019},
  publisher={alma}
}

@article{Raney:Subdirect:1953,
  title={A subdirect-union representation for completely distributive complete lattices},
  author={George N. Raney},
  journal={Proceedings of the American Mathematical Society},
  volume={4},
  number={4},
  pages={518--522},
  year={1953},
  OPTpublisher={JSTOR}
}

@article{RodriguezLopez_Romaguera:Vietoris_Hausdorff:2002,
  author={Rodr{\'\i}guez-L{\'o}pez, Jes{\'u}s and Romaguera, Salvador},
  title = {The relationship between the {Vietoris} topology and the {Hausdorff} quasi-uniformity},
  journal = {Topology and its Applications},
  volume = {124},
  number = {3},
  pages = {451--464},
  year = {2002},
  OPTissn = {0166-8641},
  doi = {10.1016/S0166-8641(01)00252-8}
}

@misc{Dagnino_Farjudian_Moggi:Robust_Topology:arXiv:2025,
  author = 	 {Francesco Dagnino and Amin Farjudian and Eugenio Moggi},
  title = 	 {Robust Topology and the {Hausdorff-Smyth} Monad on
                  Metric Spaces over Continuous Quantales},
  year={2025},
  eprint={2508.11623},
  archivePrefix={arXiv},
  primaryClass={cs.LO},
  OPTurl = {https://arxiv.org/abs/2508.11623},
  OPTdoi = {10.48550/arXiv.2508.11623}
}

@book{Johnstone86,
  title = {Stone spaces},
  publisher = {Cambridge University Press},
  year = {1986},
  author = {Johnstone, Peter T.},
  volume = {3},
  OPTpages = {xxii+370},
  series = {Cambridge Studies in Advanced Mathematics},
  address = {Cambridge},
  isbn = {0-521-33779-8},
}

@article{Vickers:Infosys:1993,
  author = {Steven Vickers},
  title = {Information systems for continuous posets},
  journal = {Theoretical Computer Science},
  volume = {114},
  number = {2},
  pages = {201--229},
  year = {1993},
  OPTissn = {0304-3975},
  doi = {10.1016/0304-3975(93)90072-2}
}

@Book{Willard:General_Topology:Book:1970,
  author={Stephen Willard},
  title = 	 {General Topology},
  publisher = 	 {Addison-Wesley},
  year = 	 {1970}
}

@book{Kelley1975,
  author    = {Kelley, John L.},
  title     = {General Topology},
  publisher = {Springer-Verlag},
  year      = {1975},
  address   = {New York, NY},
  series    = {Graduate Texts in Mathematics},
  volume    = {27}
}

\end{document}